\documentclass[12pt,openright,oneside,letterpaper,onecolumn]{report}  %% USE THIS FOR SINGLE SIDED

\newcommand{\thesistitle}{Multiple Dirichlet Series for Affine Weyl Groups}
\newcommand{\thesisauthor}{Ian Whitehead}
\newcommand{\thesisyear}{2014}

\usepackage[dvips]{epsfig}
\usepackage{fancyhdr}
\usepackage{afterpage}
\usepackage{times}
\usepackage{chngcntr}
\usepackage{graphicx}
\usepackage{latexsym}
\usepackage{amsfonts}
\usepackage{amssymb}
\usepackage{amsmath}
\usepackage{amsthm}
\usepackage{bm}
\usepackage{psfrag}
\usepackage{url}
\usepackage{subfig}
\usepackage{graphicx}
\usepackage{enumitem}
\usepackage{mathdots}
\usepackage[section]{placeins}

\usepackage{graphics}
\usepackage{enumerate, color}

\counterwithout{figure}{chapter} \counterwithout{table}{chapter}

\counterwithout{figure}{chapter} \counterwithout{table}{chapter}

\newcommand{\singlespace}{\renewcommand{\baselinestretch}{1.10} \small \normalsize}

\newcommand{\doublespace}{\renewcommand{\baselinestretch}{1.5} \small \normalsize}
\newcommand{\normalspace}{\doublespace}
\newcommand{\thesistitlepage}{
    \normalspace
    \thispagestyle{empty}
    \begin{center}
        \textbf{\LARGE \thesistitle} \\[1cm]
        \textbf{\LARGE \thesisauthor} \\[8cm]

        Submitted in partial fulfillment of the \\
        requirements for the degree \\
        of Doctor of Philosophy \\
        in the Graduate School of Arts and Sciences \\[4cm]
        \textbf{\Large COLUMBIA UNIVERSITY} \\[5mm]
        \thesisyear
    \end{center}
    \clearpage
}

\newcommand{\thesiscopyrightpage}{
    \thispagestyle{empty}
    \strut \vfill
    \begin{center}
      \copyright \thesisyear \\
      \thesisauthor \\
      All Rights Reserved
    \end{center}
    \cleardoublepage
}

\newcommand{\thesisabstract}{
    \thispagestyle{empty}
    \begin{center}
    \textbf{\large ABSTRACT} \\[1cm]
     \textbf{\large \thesistitle} \\[1cm]
     \textbf{\large \thesisauthor} \\[1cm]
    \end{center}
Let $W$ be the Weyl group of a simply-laced affine Kac-Moody Lie group, excepting $\tilde{A}_n$ for $n$ even. We construct a multiple Dirichlet series $Z(x_1, \ldots x_{n+1})$, meromorphic in a half-space, satisfying a group $W$ of functional equations. This series is analogous to the multiple Dirichlet series for classical Weyl groups constructed by Brubaker-Bump-Friedberg, Chinta-Gunnells, and others. It is completely characterized by four natural axioms concerning its coefficients, axioms which come from the geometry of parameter spaces of hyperelliptic curves. The series constructed this way is optimal for computing moments of character sums and L-functions, including the fourth moment of quadratic L-functions at the central point via $\tilde{D}_4$ and the second moment weighted by the number of divisors of the conductor via $\tilde{A}_3$. We also give evidence to suggest that this series appears as a first Fourier-Whittaker coefficient in an Eisenstein series on the twofold metaplectic cover of the relevant Kac-Moody group. The construction is limited to the rational function field $\mathbb{F}_q(t)$, but it also describes the $p$-part of the multiple Dirichlet series over an arbitrary global field.
    \cleardoublepage
}

\newtheorem{theorem}{Theorem}[section]
\newtheorem{thm}[theorem]{Theorem}

\newtheorem{cor}[theorem]{Corollary}

\newtheorem{lemma}[theorem]{Lemma}

\newtheorem{prop}[theorem]{Proposition}

\newtheorem{conj}[theorem]{Conjecture}
\newtheorem{property}[theorem]{Property}
\newtheorem{axiom}[theorem]{Axiom}
\newtheorem{condition}[theorem]{Condition}
\theoremstyle{definition}

\theoremstyle{remark}

\numberwithin{equation}{section}
\newtheoremstyle{dotless}{}{}{}{}{\bfseries}{}{ }{}
\theoremstyle{dotless}

\newcommand{\C}{\mathbb{C}}

\newcommand{\Q}{\mathbb{Q}}
\newcommand{\Z}{\mathbb{Z}}

\newcommand{\F}{\mathbb{F}}

\newcommand{\x}{{\mathbf {x}}}

\newcommand{\res}[2]{\left(\frac{#1}{#2}\right)}

\newcommand{\sgn}{{\mathrm {sgn}}}

\renewcommand{\tilde}{\widetilde}

\begin{document}
\pagestyle{empty}

\thesistitlepage 

\thesiscopyrightpage

\thesisabstract

% In the "roman-numbered" section of the thesis, we have numbers at the bottom
% and we have to reduce the textheight of the text to make space for the number
\pagenumbering{roman} \pagestyle{plain}

\setcounter{tocdepth}{1}
\renewcommand{\contentsname}{Table of Contents}
\tableofcontents 

\cleardoublepage

\listoffigures
\addcontentsline{toc}{section}{List of Figures}
\cleardoublepage

%\listoftables 
%\cleardoublepage

%Acknowledgements
\chapter*{Acknowledgements}
This dissertation would not exist without the guidance of two exceptional advisors: Dorian Goldfeld and Adrian Diaconu. Adrian shared his vision of axiomatic Kac-Moody multiple Dirichlet series with me, which became the foundation of this project. He was endlessly generous with his time and brilliant ideas, sharing them in Providence, Banff, and Minneapolis, on freezing cigarette breaks and over plates of buffalo wings. Dorian offered his unparalleled mathematical perspective and wisdom, urging me always to simplify the problem and hone in on the essential details. He was incredibly supportive--I could go into his office in despair, and emerge feeling on the cusp of a breakthrough. I thank Gautam Chinta for many thoughtful suggestions on the project, and for valuable comments on drafts of this paper. I had enlightening conversations with Jeff Hoffstein, Ben Brubaker, Kyu-Hwan Lee, Manish Patnaik, and Jordan Ellenberg, among many others. Anna Pusk\'{a}s suggested the crucial step in the proof in Chapter 6, and was a much-needed sounding board and a great friend throughout this process. I am grateful to my fellow graduate students at Columbia University, and to the participants in a special semester program at ICERM on ``Automorphic Forms, Combinatorial Representation Theory and Multiple Dirichlet Series,'' where much of this work took shape. 
\cleardoublepage

%%%
%%% BODY
%%%
\pagestyle{headings} \pagenumbering{arabic}

%
% In the "arabic" section of the thesis, we do not have numbers at  the
% bottom and we want to use the full length of the page to avoid vbox
% underfulls. We use the fancyheaders package to adapt the headers
% according to the  Columbia requirements.
% EDIT: We changed the placement of page numbering to be at bottom centre - might need respacing later. PA
\setlength{\textheight}{8.5in} \setlength{\footskip}{.5in}

% We change the pagestyle
\fancypagestyle{plain} {%
\fancyhf{} \fancyfoot[C]{\thepage} \fancyhead[RE,LO]{\itshape \leftmark}
\renewcommand{\headrulewidth}{0pt}
} \pagestyle{plain}

%Title/short description of chapter
\chapter{Introduction}\label{chap:introduction}
\section{Motivations}

The goal of this paper is to define and construct multiple Dirichlet series associated to affine Kac-Moody Lie groups over the rational function field $\F_q(t)$. These will be power series in several complex variables, with meromorphic continuation to a half-space, and an infinite group of symmetries isomorphic to the Kac-Moody Weyl group. They generalize the multiple Dirichlet series for finite Weyl groups which have been thoroughly studied elsewhere. However, there are important distinctions between the finite and Kac-Moody cases, necessitating a new approach. This paper tests one such approach, and makes a first foray into new territory for multiple Dirichlet series.

The original goal of multivariable Dirichlet series is to parametrize a family of L-functions. The technique of studying L-functions on average across a family has yielded tremendous progress, from the Bombieri-Vinogradov theorem to recent work of Bhargava and his collaborators. One primary focus has been on $r$th moment problems, concerning averages (or weighted averages) of L-functions at the central point
\begin{equation}
\frac{\sum_{L \in \mathcal{F}} L(1/2)^r}{\sum_{L \in \mathcal{F}} 1}
\end{equation}
considered in limit as the size of the family $\mathcal{F}$ approaches infinity. The Katz-Sarnak philosophy, that random L-functions are modeled by characteristic polynomials of random matrices, leads to moment conjectures in a broad variety of cases \cite{CFKRS, KS}. There are many approaches to proving these conjectures, including trace formulas, approximate functional equations, and multiple Dirichlet series. Typically, the first few moments can be computed, but higher moment conjectures remain open. 

The first example of a multiple Dirichlet series, now understood as the $A_2$ series, is essentially
\begin{equation}
Z(s, t) = \sum_{m, n} \res{m}{n} m^{-s}n^{-t} = \sum_{m} L(t, \chi_m) m^{-s}
\end{equation}
where $\res{\,}{\,}$ denotes the quadratic residue symbol, and $\chi_m$ the equivalent quadratic character. This appears in a different guise in the work of Siegel \cite{S}. Goldfeld and Hoffstein described it as the Mellin transform of an Eisenstein series of half-integral weight on GL(2), and used it to compute the first moment in the family of quadratic L-functions \cite{GH}. Other series compute the second \cite{BFH3} and third \cite{DGH} moments. These moments were originally computed in other ways, by Jutila for $r=1, 2$ \cite{J} and Soundararajan for $r=3$ \cite{So}. For other moment computations via multiple Dirichlet series, see \cite{BrFH}, \cite{BFH1}, and \cite{BFH2}. However, regardless of the strategy, the fourth moment seems out of reach. From the multiple Dirichlet series perspective, the difficulty is transitioning from finite to infinite Weyl groups of functional equations. In \cite{BD}, Bucur and Diaconu construct a series with affine Weyl group $\tilde{D}_4$ of functional equations, and use it to compute the fourth moment of quadratic L-functions in the rational function field $\F_q(t)$. Theirs is the only affine Weyl group multiple Dirichlet series currently in the literature, and this paper generalizes their work. The constructions here are still limited to the rational function field, but generalized to arbitrary affine Weyl groups, which allow many new moment computations of equivalent difficulty. These are some of the first applications of Kac-Moody groups to number theory.

There is now a rich literature that treats multiple Dirichlet series as objects of intrinsic interest, aside from applications to analytic number theory. To guarantee the desired functional equations, one must replace certain sums of L-functions with weighted sums; the problem of choosing weights leads to important questions in combinatorial representation theory. Weyl group multiple Dirichlet series are constructed in papers of Brubaker, Bump and Friedberg \cite{BBF1, BBF2}, and Chinta and Gunnells \cite{CG1, CG2}. The input is a global field $k$, an integer $n$, and a root system with Weyl group $W$; the output is a multivariable Dirichlet series constructed out of Gauss sums for order $n$ characters on $k$. The series has a finite group of functional equations isomorphic to $W$. Its ``$p$-part,'' which describes its weighting, can be interpreted as a sum on a crystal, or as a deformation of the Weyl character formula for $W$. One hopes that all of this machinery can eventually be extended to affine and arbitrary Kac-Moody Weyl groups, with the present work (which has $k=\F_q(t)$ and $n=2$) as a template.

The Eisenstein conjecture explains the link between multiple Dirichlet series and automorphic forms. It states that each Weyl group multiple Dirichlet series appears as the first Fourier-Whittaker coefficient in an Eisenstein series on the metaplectic $n$-fold cover of the algebraic group over $k$ associated to the root system. This is proven for root systems of type A \cite{BBF1} and B \cite{FZ}. It is too early to state a generalization of this conjecture to Kac-Moody groups. Eisenstein series on metaplectic covers of Kac-Moody groups have not yet been constructed. However, nonmetaplectic Kac-Moody Eisenstein series are now an area of active research. Recent work of Braverman, Garland, Kazhdan, Miller, and Patnaik makes progress constructing Eisenstein series on affine Kac-Moody algebras over function fields \cite{BK, BGKP, G, GMP}. Completely understanding one Whittaker coefficient of a conjectural Eisenstein series could shed light on the series as a whole. For example, the analytic behavior of a Whittaker coefficient models the behavior of the full Eisenstein series via Maass-Selberg type relations. There is some evidence to suggest that the series presented here are the correct ones for a hypothetical generalization of the Eisenstein conjecture.

\section{Methods and Main Results}

Let $q$ be a prime power congruent to $1$ modulo $4$. Let $W$ be a simply-laced affine Weyl group represented by one of the Dynkin diagrams:

\includegraphics[scale=.3]{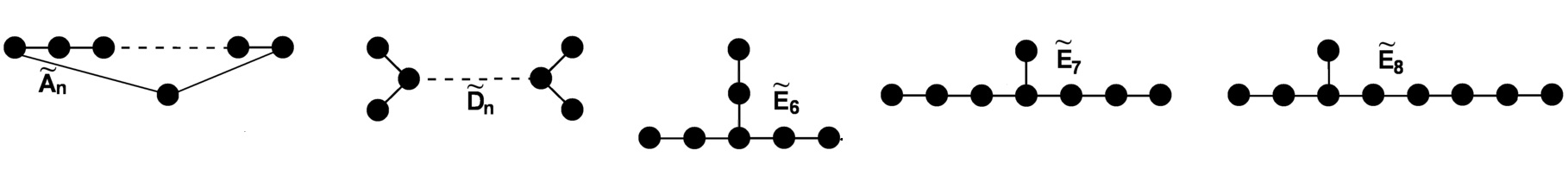}

\noindent and label the vertices $1$ to $n+1$. If the root system is type $\tilde{A}_n$, then we must assume for technical reasons that $n$ is odd. Write $i\sim j$ if vertices $i$ and $j$ are adjacent. Then the quadratic $W$ multiple Dirichlet series over the rational function field $\F_q(t)$ is roughly:
\begin{equation}\label{Z}
Z(x_1, \ldots x_{n+1})=\sum_{f_1, \ldots f_{n+1} \in \mathbb{F}_q[t] \text{ monic}} \left( \prod_{i \sim j} \left(\frac{f_i}{f_j}\right)\right) x_1^{\deg f_1}\cdots x_{n+1}^{\deg f_{n+1}}
\end{equation}
where we have replaced the usual variable $q^{-s_i}$ with $x_i$. To make this precise, the product of residue symbols should be replaced by a certain weighted term in cases where the $f_i$ are not squarefree or coprime. The main theorem of this paper is that the choice of weights and the resulting series are uniquely determined by four natural axioms (\ref{twistedmult}, \ref{localglobal}, \ref{dominance}, \ref{initialconditions}). The series has meromorphic continuation to a half-space, with group of functional equations $W$. In the case of $\tilde{A}_n$, we prove meromorphicity in the largest possible domain, which corresponds to the Tits cone of $W$; in other types, we only prove meromorphicity in a smaller half space, but we give a conjecture which implies meromorphic continuation to the optimal domain. 

There are other multivariable functions with the same domain of meromorphic continuation and the same group of functional equations, but the series constructed here is optimal for computing analytic data on character sums and L-functions. Because of its natural axioms, which arise from algebraic geometry, and because it has certain poles which will be described below, this series also seems like the correct one to satisfy the Eisenstein conjecture--that is, to be a Whittaker coefficient in a metaplectic Kac-Moody Eisenstein series. The theorem is limited to the rational function field $\mathbb{F}_q(t)$, but the proof resolves all combinatorial questions involved in constructing analogous series over any global field; meromorphic continuation is the only remaining obstacle, and this is known to be an extremely difficult problem. 

To contrast this theorem with the case of finite Weyl groups: for finite $W$, the series $Z$ is completely determined by its expected functional equations; it has meromorphic continuation to all of $\mathbb{C}^n$, and in fact is a rational function in the variables $x_i$. For affine Weyl groups, all of this breaks down. The series cannot be rational because infinitely many functional equations mean infinitely many poles; moreover the poles accumulate at essential singularities along the boundary of the Tits cone. The series is not uniquely determined by its functional equations--we will show that it is determined up to a meromorphic function of one variable. One needs a completely new strategy to choose the correction terms in a canonical way. Lee and Zhang generalize the averaging method of Chinta and Gunnells to construct a series for every symmetrizable Kac-Moody algebra \cite{LZ}. This has the desired functional equations, but it does not naturally contain character sums or L-functions. Bucur and Diaconu construct their $\tilde{D}_4$ series by making an assumption about its residue at one pole \cite{BD}. Their series is closely related to the one constructed here, and satisfies the first two axioms. However, it does not satisfy the third axiom, and likely will not fulfill the Eisenstein conjecture. 

In order to state the axioms I will use some additional notation. In the definition of the series (\ref{Z}), we replace $\prod_{i \sim j} \left(\frac{f_i}{f_j}\right)$ with $H(f_1, \ldots f_{n+1})\in \C$, which includes the weights mentioned above. Let $c_{a_1, \ldots a_{n+1}}(q) \in \C$ be a power series coefficient of $Z$, so that 
\begin{equation}
c_{a_1, \ldots a_{n+1}}(q)=\sum_{\substack{f_i \in \mathbb{F}_q[t] \text{ monic,} \\ \deg(f_i)=a_i}} H(f_1, \ldots f_{n+1}).
\end{equation}
The axioms concern the behavior of the weights $H$ and coefficients $c$ as the underlying finite field $\F_q$ varies. The first axiom is twisted multiplicativity for the $H$ terms: if we assume that $\gcd(f_1\cdots f_{n+1}, f_1'\cdots f_{n+1}')=1$, then
\begin{equation}
H(f_1 f_1', \ldots f_{n+1} f_{n+1}')=H(f_1, \ldots f_{n+1}) H(f_1', \ldots f_{n+1}')\prod_{i \sim j} \left(\frac{f_i}{f_j}\right)
\end{equation}
so it suffices to describe $H(p^{a_1}, \ldots p^{a_{n+1}})$ for $p$ prime. This condition is familiar from the theory for finite Weyl groups, but the next two axioms are new, hypothesized by Diaconu and Pasol in a forthcoming paper \cite{DP}. They can be proved as propositions in the finite Weyl group case. The second axiom is a local-to-global property: the terms $c_{a_1, \ldots a_{n+1}}(q)$ and $H(p^{a_1}, \ldots p^{a_{n+1}})$ are polynomials in $q$ and $|p|:=q^{\deg p}$ respectively, and 
\begin{equation}
H(p^{a_1}, \ldots p^{a_{n+1}})=|p|^{a_1+\cdots +a_{n+1}} c_{a_1, \ldots a_{n+1}}(1/|p|).
\end{equation}
The third axiom is a dominance condition: $c_{a_1, \ldots a_{n+1}}(q)$ has nonzero terms only in degrees  $(a_1+\cdots+a_{n+1}+1)/2<d\leq a_1+\cdots+a_{n+1}$, so the degree of $H(p^{a_1}, \ldots p^{a_{n+1}})$ is less than $(a_1+\cdots+a_{n+1}-1)/2$. This means in practice that the contribution of correction factors in the character sums $c$ is as small as possible. The final axiom is just a normalization condition, that $H(1, \ldots 1, f, 1, \ldots 1)=1$.

We will briefly outline the geometric meaning of Axioms 2 and 3. For more details, we refer the reader to the work of Diaconu and Pasol. The sums 
\begin{equation}
\sum_{\substack{f_i \in \mathbb{F}_q[t] \text{ monic,}} \\ \deg(f_i)=a_i} \left( \prod_{i \sim j} \left(\frac{f_i}{f_j}\right)\right)
\end{equation}
count points on a certain variety over $\F_q$; for each $i \in \lbrace 1, \ldots n+1 \rbrace$, this variety is a cover of the parameter space of hyperelliptic curves $y^2=\prod_{j\sim i} f_j(t)$. The weights $H(f_1, \ldots f_{n+1})$ are best understood as  counting points when this variety is desingularized and compactified. Axioms 2 and 3 translate into statements about the cohomology of the nonsingular, compact variety: Axiom 2 is a duality statement, and Axiom 3 is a cohomological purity statement. 

Diaconu and Pasol study series corresponding to Dynkin diagrams of the form:

\includegraphics[scale=.2]{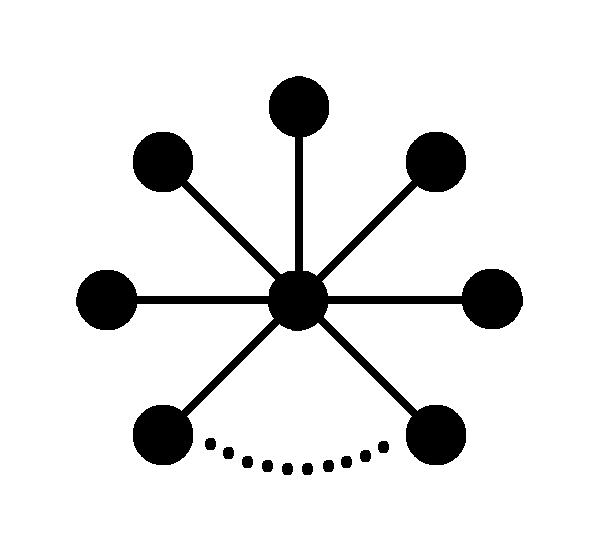}

\noindent If meromorphically continued, these series would compute the $n$th moment of quadratic Dirichlet L-functions over function fields. The diagram gives a Kac-Moody Lie group for $n\geq 4$, and the groups become extremely complicated as $n$ grows. For concreteness, consider the coefficient $c_{1,\ldots 1, 3}(q)$ in such a series. This is roughly:
\begin{equation}
\sum_{\deg f_1=\ldots=\deg f_n=1, \deg f_{n+1}=3} \left( \frac{f_{n+1}}{f_1\cdots f_n} \right)=\sum_{\deg f_{n+1}=3}(\sum_{x \in \mathbb{F}} \chi(f_{n+1}(x)))^n
\end{equation}
or the $n$th moment of the trace of Frobenius in the family of elliptic curves $y^2=f_{n+1}(t)$ over the fixed finite field $\F_q$. Birch studies these moments via the Eichler-Selberg trace formula \cite{B}. The first 9 moments are polynomials in $q$, but the 10th is not--it includes the Ramanujan tau function. This corresponds to the appearance of higher-dimensional Galois representations in the cohomology of the underlying variety. In such cases, the axioms must be relaxed so that $H$ and $c$ are no longer polynomials, but still consist of terms which have well-defined degree and can be evaluated at negative powers of a prime $p$. For example, the degree of $\tau$ is taken to be $\frac{11}{2}$, and we set $\tau(p^{-1})=p^{-11}\tau(p)$. The underlying geometry here is not fully understood. For example, what are the simplest Kac-Moody Weyl groups for which nonpolynomial coefficients must appear? Polynomial coefficients are sufficient for affine Weyl groups, but this intriguing question remains unanswered.

The first tool in the proof of the main theorem is a detailed study of the functional equations of $Z(x_1, \ldots x_{n+1})$, which are verified directly from the axioms. The functional equations induce recursive formulas for the coefficients of the series; this recursion can be solved up to the choice of a one-parameter family of diagonal coefficients $c_{m\alpha_0}$, where $m \in \Z_{\geq 0}$ and $\alpha_0$ is the minimal imaginary root of the affine root system. The diagonal coefficients are analogous to the central coefficients in function field Dirichlet L-functions; they are the most difficult character sums to compute by hand. Determining them uniquely requires the full strength of Axioms 2 and 3. In the diagonal coefficients, it is possible to observe the effect of imaginary roots on the multiple Dirichlet series, and, one hopes, on the Kac-Moody Eisenstein series.

The second tool in the proof is to take residues of $Z(x_1, \ldots x_{n+1})$, setting various coefficients $x_i=q^{-1}$. Crucially, certain residues of the series are Euler products, with multiplicative rather than twisted multiplicative coefficients. This simplifies the series enough to be written down explicitly, as an infinite product of function field zeta functions. The local-to-global property leads to a symmetry in the residue, and dominance together with the functional equations determines it uniquely. We compute the full residue in type $\tilde{A}_n$ with $n$ odd, and compute the residue up to a diagonal factor in all types. For example, in the case of $\tilde{A}_3$, with the vertices of the Dynkin diagram labeled from $1$ to $4$ cyclically, we prove the following residue formula:
\begin{align} \label{A3residue}
(-q)^{(n+1)/2} \text{Res}_{x_2=x_4=q^{-1}} Z(x_1, x_2, x_3, x_4)=\prod \limits_{m=0}^{\infty}
&(1-x_1^{2m+2}x_3^{2m})^{-1}(1-qx_1^{2m+2}x_3^{2m})^{-1} \nonumber \\
&(1-x_1^{2m}x_3^{2m+2})^{-1}(1-qx_1^{2m}x_3^{2m+2})^{-1} \nonumber \\
&(1-x_1^{2m+2}x_3^{2m+2})^{-2}(1-qx_1^{2m+2}x_3^{2m+2})^{-2} \nonumber \\
&(1-x_1^{2m+1}x_3^{2m+1})^{-1}(1-qx_1^{2m+1}x_3^{2m+1})^{-1}
\end{align}
The full series can be recovered from the residue. 

Let us sketch a possible arithmetic application of this formula. We may interpret the $\tilde{A}_3$ series as roughly:
\begin{equation}
\sum_{f_1, f_2, f_3, f_4} \res{f_1 f_3}{f_2 f_4} x_1^{-\deg f_1}\cdots x_4^{-\deg f_4} = \sum_{f_2, f_4} L(x_1, \chi_{f_2 f_4})L(x_3, \chi_{f_2 f_4}) x_2^{-\deg f_2} x_4^{-\deg f_4}.
\end{equation}
We set $x_2=x_4=x$, multiply by $x^{-d-1}$ and take $\frac{1}{2 \pi i} \int_{|x|=\epsilon}$. This integral can be evaluated by expanding the circle $|x|=\epsilon$ across the pole $x=q^{-1}$, where we gain the residue \ref{A3residue}. We obtain a formula for the sum
\begin{equation}
\sum_{\deg{f}=d}\sigma_0(f)L(x_1, \chi_{f})L(x_3, \chi_{f})
\end{equation}
where $\sigma_0(f)$ is the number of divisors of $f$. Evaluating at $x_1=x_3=q^{-1/2}$ gives the second moment of quadratic L-functions over $\F_q(t)$ with conductor of degree $d$, weighted by the number of divisors of the conductor. It is possible to sieve for squarefree conductors as well.

The residue also contains evidence related to the Eisenstein conjecture. The first five factors in (\ref{A3residue}) correspond to positive real roots in the $\tilde{A}_3$ root system. The last three factors, however, correspond to imaginary roots. Eisenstein series on Kac-Moody algebras, and hence their Whittaker functions, are expected to have poles corresponding to all roots, real and imaginary. The contribution of imaginary roots is subtle and difficult to detect. We cannot completely determine the poles of $Z(x_1, \ldots x_{n+1})$ corresponding to imaginary roots, since some of them may be canceled out in the residue. However, we can assert that such poles exist. They do not appear in the Bucur-Diaconu $\tilde{D}_4$ series. Their presence here suggests that the four axioms produce series which could fulfill the Eisenstein conjecture. 

\section{Further Directions}

The first problem arising from this work is to prove Conjecture \ref{R1}, which gives explicit formulas for the diagonal parts of the residues in all types.  Meromorphic continuation of the series $Z(x_1, \ldots x_{n+1})$ to its largest possible half-space will follow immediately. A second task is to prove the main theorem in type $\tilde{A}_n$ with $n$ even. The difficulty here is that the method of studying a residue with an Euler product formula may not apply. This should be an inconvenience rather than a fundamental obstruction; one can still study a residue whose terms satisfy a very simple twisted multiplicativity property. A third natural extension of the theorem is to affine root systems which are not simply-laced: $\tilde{B}_n$, $\tilde{C}_n$, $\tilde{F}_4$, and $\tilde{G}_2$. 

To generalize further, following the theory for finite Weyl groups, one could replace the quadratic residue symbols in (\ref{Z}) with $m$th power residue symbols, or Gauss sums; in the Eisenstein conjecture, this means constructing a Whittaker function on the $m$-fold metaplectic cover of the Kac-Moody algebra. It remains to be seen whether the four axioms, suitably modified, still yield a canonical choice for this series. A separate project is to generalize the construction to arbitrary global fields. In this case the $p$-parts of the series, i.e. the weights $H(p^{a_1}, \ldots p^{a_{n+1}})$, match the rational function field construction, but the global series is quite different. The residue formulas should generalize straightforwardly to any global field, with the Dedekind zeta function of the field replacing the function field zeta function $(1-qx)^{-1}$. One could consider arbitrary function fields, following the work of Hoffstein and Rosen \cite{HR} and Fisher and Friedberg \cite{FF1, FF2}. The multiple Dirichlet series will still be power series, now with finitely many one-parameter families of undetermined coefficients. It is possible that meromorphic continuation of the residue is enough to imply meromorphic continuation of the full series. Unfortunately, this line of reasoning breaks down completely over number fields. Over $\Q$, for example, meromorphic continuation seems out of reach at present; it may not be proven until the theory of Kac-Moody Eisenstein series is fully developed.

Finally, this work could be generalized towards non-affine Kac-Moody groups. These are necessary to compute fifth and higher moments of quadratic L-functions. Here the situation is very complex. The functional equations have even less control over the shape of the series, but the Diaconu-Pasol axioms should still guarantee uniqueness in many cases. One fundamental question, discussed above, is to determine when the coefficients of the series will no longer be polynomials. This would create major obstacles to the methods of this paper. On the other hand, better-understood Kac-Moody algebras, like hyperbolic ones, may still generate polynomial coefficients, and could be tractable. One could also attempt to study properties of Kac-Moody Whittaker functions abstractly--for example, what differential equations must they satisfy?--and thereby accumulate further evidence for the Eisenstein conjecture without actually constructing Kac-Moody Eisenstein series. 

\section{Background and Notation: Function Field Dirichlet L-Functions}

Let $q$ be an odd prime power, $\F_q$ a finite field, and $\F_q[t]$ its polynomial ring. 

We will call an element $p$ of $\F_q[t]$ prime if it is monic, nonconstant, and irreducible. For $p\in \F_q[t]$ prime, and any $g \in \F_q[t]$ we define the quadratic residue symbol, or quadratic character modulo $p$, as
\begin{equation}
\res{p}{g} = \chi_p(g):= \left\lbrace \begin{array}{ll} 1 & g \text{ square modulo } p \\ -1 & g \text{ not a square modulo } p \\ 0 & g \equiv 0 \text{ modulo } p. \end{array} \right.
\end{equation} 
For any nonzero $f\in \F_q[t]$, we define $\sgn(f)$ to be $1$ if the leading coefficient of $f$ is a square in $\F_q^*$, and $-1$ if it is not a square. For $f \in \F_q^*$ constant, we set
\begin{equation}
\res{f}{g} = \chi_f(g):= \sgn(f)^{\deg g}.
\end{equation} 
These are multiplicative functions of $g$; we extend to arbitrary $f \in \F_q[t]$ by multiplicativity as well:
\begin{equation}
\res{f_1f_2}{g}=\res{f_1}{g}\res{f_2}{g}.
\end{equation}
Then we have the all-important quadratic reciprocity law:
\begin{equation}
\res{f}{g}=(-1)^{(q-1)(\deg f)(\deg g)/2}\sgn(f)^{\deg g}\sgn(g)^{\deg f} \res{g}{f}
\end{equation}
or if $q \equiv 1 \mod 4$ (which we will assume below) and $f, g$ monic, simply
\begin{equation}
\res{f}{g}=\res{g}{f}.
\end{equation}

Next, we define zeta and L-functions over $\F_q[t]$. These are typically written as series in the variable $q^{-s}$ to highlight parallels with L-functions over number fields, but for our purposes it is more convenient to use the variable $x$. Let 
\begin{equation} 
\zeta(x):=\sum_{g\in \F_q[t] \text{ monic}} x^{\deg g} = \prod_{p\in \F_q[t] \text{ prime}} (1-x^{\deg p})^{-1}.
\end{equation}
This zeta function may be computed explicitly as $\zeta(x)=(1-qx)^{-1}$, since there are $q^d$ monic polynomials of degree $d$. Hence we automatically have meromorphic continuation to all $x\in \C$, with a functional equation 
\begin{equation}
(1-x)^{-1}\zeta(x)=q^{-1}x^{-2} (1-q^{-1}x^{-1})\zeta(q^{-1}x^{-1}).
\end{equation}
We also have the Riemann hypothesis, trivially, since $\zeta(x)$ has no zeroes at all.

For $f\in \F_q[t]$ squarefree, define the quadratic Dirichlet L-function with conductor $f$ as 
\begin{equation} \label{L-function}
L(x, \chi_f):=\sum_{g \in \F_q[t] \text{ monic}} \chi_f(g) x^{\deg g} = \prod_{p\in \F_q[t] \text{ prime}} (1-\chi_f(p)x^{\deg p})^{-1}.
\end{equation}
If $f \in \F_q^*$ is constant, then we have $L(x, \chi_f)=\zeta(\sgn(f)x)$. Otherwise, $L(x, \chi_f)$ is a polynomial in $x$, whose degree is $\deg f -1$. This follows from the orthogonality relation for the nontrivial character $\chi_f$. We have a functional equation as follows: if $\deg f$ is odd, then
\begin{equation}
L(x, \chi_f) = (q^{1/2}x)^{\deg f -1} L(q^{-1}x^{-1}, \chi_f),
\end{equation}
and if $\deg f$ is even, then
\begin{equation}
(1-\sgn(f)x)^{-1}L(x, \chi_f)= (q^{1/2}x)^{\deg f -2} (1-\sgn(f)q^{-1}x^{-1})^{-1}L(q^{-1}x^{-1}, \chi_f).
\end{equation}
This functional equation is a consequence of the Weil conjectures for curves, since one can show that $\zeta(x)L(x, \chi_f)=\zeta_C(x)$, where $C$ is the hyperelliptic curve $y^2=f(t)$. We also have the Riemann hypothesis for $L(x, \chi_f)$: all roots have $|x|=q^{-1/2}$. This implies a bound of $\binom{\deg f -1}{a} q^{-a/2}$ on the $x^a$ coefficient of $L(x, \chi_f)$. All these properties are developed in detail in \cite{R}. For the proof of the Weil conjectures for all varieties over finite fields, see the work of Deligne \cite{D1, D2}.

If $f \in \F_q[t]$ is not squarefree, then the definition (\ref{L-function}) gives only a partial Euler product. It is more natural to take the L-function of the character $\chi_{f_0}$, where $f_0$ is the squarefree part of $f$. However, this means that the degree of the L-function is no longer the expected $\deg f -1$, and the functional equation is correspondingly different. This motivates the need for weighted sums of L-functions and weighted sums of quadratic characters, which are introduced below.

\section{Background and Notation: Affine Root Systems}

The representation theory and combinatorics of affine Kac-Moody root systems is a vast and important subject. We will confine our discussion to properties of the Weyl group and root system used in this paper. For a more general introduction to the subject, see \cite{H}. For complete details, see \cite{Bo}. One particularly beautiful and relevant application of this theory is to the Macdonald identities for the affine Weyl denominator \cite{M}.

We begin with one of the simply-laced affine Dynkin diagrams $\tilde{A}_n$, $\tilde{D}_n$, $\tilde{E}_6$, $\tilde{E}_7$, or $\tilde{E}_8$, with vertices labeled from $1$ to $n+1$. For convenience, we fix the following labelings:

\includegraphics[scale=.5]{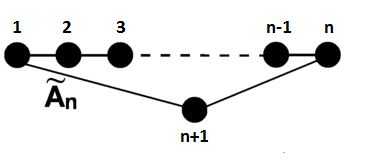} \includegraphics[scale=.5]{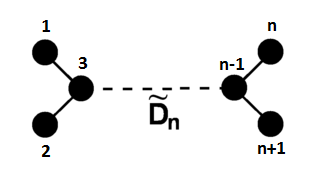}

\includegraphics[scale=.5]{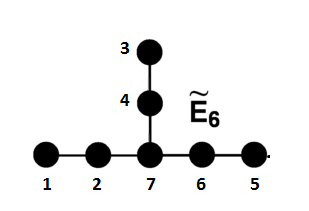} \includegraphics[scale=.5]{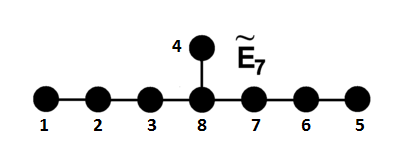}

\includegraphics[scale=.5]{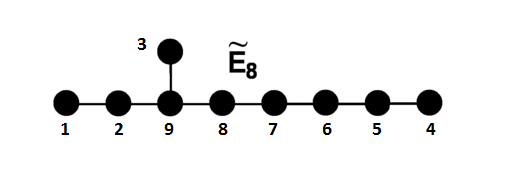}

We regard the labels as integers modulo $n+1$ in the $\tilde{A}_n$ case. Write $i\sim j$ for adjacent (distinct) vertices. Then the affine Weyl group associated to the diagram is 
\begin{equation}
W=<\sigma_1, \ldots \sigma_{n+1}\, : \, \sigma_i^2=1, \, \sigma_i\sigma_j\sigma_i=\sigma_j\sigma_i\sigma_j \text{ for } i\sim j, \, \sigma_i\sigma_j=\sigma_j \sigma_i \text{ for }i \not\sim j>.
\end{equation}
This is a Coxeter group which can be constructed as the semidirect product of a finite Weyl group with its coroot lattice. For $w \in W$, $\ell(w)$ is the length of any minimal expression $w=\sigma_{i_1} \sigma_{i_2}\cdots \sigma_{i_\ell(w)}$.

The Weyl group acts naturally on a vector space spanned by simple roots $\lbrace e_1, \ldots e_{n+1} \rbrace$. The action is defined as follows:
\begin{equation}
\sigma_i(e_j)=\left\lbrace \begin{array}{cc} -e_j & j=i \\ e_j+e_i & i \sim j \\ e_j & i\not\sim j \end{array} \right.
\end{equation}
and extended by linearity. One can check that this is a well-defined left action. The set of (real) roots $\Phi$ is the orbit of the simple roots under this action. It is an infinite set, contained in the root lattice $\Lambda=\bigoplus_{i=1}^{n+1}\Z e_i$. 

The height of a root $\alpha=\sum a_i e_i$ is $\text{ht}(\alpha)=\sum a_i$. We have a partial ordering on the roots: for $\alpha=\sum a_i e_i$, $\beta=\sum b_i e_i \in \Phi$, write $\alpha \leq \beta$ if all $a_i \leq b_i$. Define the set $\Phi^+$ of positive roots as $\lbrace \alpha \geq 0 \rbrace$ and the set $\Phi^-$ of negative roots as $\lbrace \alpha \leq 0 \rbrace$. It is a fact that every root is either positive or negative, and multiplication by $-1$ is an involution of $\Phi$. Furthermore, for $w \in W$, we may define $\Phi(w):= \Phi^+ \cap w^{-1}(\Phi^-)$. This is a finite set, whose cardinality is $\ell(w)$.

There is a unique linear subspace of the root space which is invariant under the action of the Weyl group. It is possible to find a minimal positive vector of the root lattice, $\alpha_0$, which lies in this subspace. We refer to $\alpha_0$ as the imaginary root, and the set of integer multiples of $\alpha_0$ as the set of imaginary roots. The following table gives the coordinates of $\alpha_0$ in each type:

\begin{tabular}{l | r}
Type & $\alpha_0$ \\
\hline
$\tilde{A}_n$ & $(1, 1, \ldots 1)$ \\
$\tilde{D}_n$ & $(1, 1, 2, 2, \ldots 2, 1, 1)$ \\
$\tilde{E}_6$ & $(1, 2, 1, 2, 1, 2, 3)$ \\
$\tilde{E}_7$ & $(1, 2, 3, 2, 1, 2, 3, 4)$ \\
$\tilde{E}_8$ & $(2, 4, 3, 1, 2, 3, 4, 5, 6)$ \\
\end{tabular}

\noindent The imaginary roots are analogous to real roots, but there are many distinctions. When we refer to ``roots'' below, we mean real roots only. Our convention is that $\alpha_0 \not\in \Phi$. The role of imaginary roots will almost always have to be treated separately.

The following proposition classifies real roots:
\begin{prop}
Let $\Psi$ be the finite set of roots $\alpha$ with $0\leq \alpha \leq \alpha_0$. Then the full set of roots is $\lbrace \alpha + m\alpha_0 : \alpha \in \Psi, m \in \Z \rbrace$. 
\end{prop}
This proposition is essential to several meromorphic continuation proofs given below. It demonstrates the importance of the imaginary root to the theory of affine root systems.

Another proposition describes the action of $W$ on $\Phi$ more concretely.
\begin{prop} \label{dist}
Let $\alpha= \sum a_i e_i \in \Phi$. Then $\sigma_j(\alpha)=\alpha \pm e_j$ if $\sum_{i\sim j} a_j$ is odd, and $\sigma_j(\alpha)=\alpha$ if $\sum_{i\sim j} a_j$ is even. The only exception is for $\alpha=\pm e_j + m \alpha_0$: in this case, $\sigma_j(\alpha)=\alpha \mp 2e_j$.
\end{prop}

Let $\x=(x_1, \ldots x_{n+1})$ be an $n+1$-tuple of complex numbers. In order to define functional equations below, we will fix an action of $W$ on $\x$, given by 
\begin{equation}
(\sigma_i(\x))_j=\left\lbrace \begin{array}{cc} q^{-1}x_i^{-1} & j=i \\ q^{1/2}x_ix_j & i \sim j \\ x_j & i\not\sim j \end{array} \right. .
\end{equation}
The action of $w\in W$ takes the monomial $\x^{\alpha}:=\prod x_i^{a_i}$ to $q^{(\text{ht}(w(\alpha))-\text{ht}(\alpha))/2}\x^{w(\alpha)}$. It is a conjugation by $q^{1/2}$ of the standard action of $W$ on the group ring of the root lattice $\C[\Lambda]\cong\C[x_1, \ldots x_{n+1}]$.

%Title/short description of chapter
\chapter{Axioms and Consequences}\label{chap:2}
\section{The Four Axioms}

We will begin by stating the problem in the most general way. Let $\Gamma$ be a graph with vertices labeled $1$ to $n+1$, and write $i \sim j$ for adjacent vertices. Let $\mathbb{F}_q$ be a finite field, with $q\equiv 1 \mod 4$. We would like to define a multiple Dirichlet series, roughly:
\begin{equation} 
\sum_{f_1, \ldots f_{n+1} \in \mathbb{F}_q[t] \text{ monic}} \prod_{i \sim j} \left(\frac{f_i}{f_j}\right) x_1^{\deg f_1} \cdots x_{n+1}^{\deg f_{n+1}}
\end{equation}
where $\left(\frac{\text{  }}{\text{  }}\right)$ denotes the quadratic residue symbol. This series should have a group of functional equations isomorphic to $W$, the Weyl or Coxeter group associated to the simply-laced Dynkin diagram $\Gamma$. The situation where $\Gamma$ corresponds to a finite irreducible root system is well-understood; we will restrict our attention to the case of affine Kac-Moody root systems below. In any case, the functional equations derive from those of quadratic Dirichlet L-series for the function field $\mathbb{F}_q(t)$. We expect to have a functional equation in $x_i \mapsto q^{-1}x_i^{-1}$ based on the L-function of a quadratic character with conductor $\prod\limits_{j \sim i} f_j$. However, if the conductor is not squarefree, the functional equation is different from what we expect; we must remedy this situation by replacing the sums of residue symbols above with weighted sums. To this end, we define a weighting function $H: \F_q[t]^{n+1} \to \C$, and let 
\begin{align}
Z(x_1, \ldots x_{n+1}) &= \sum_{f_1, \ldots f_{n+1} \in \mathbb{F}_q[t] \text{ monic}} H(f_1, \ldots f_{n+1}) x_1^{\deg f_1} \cdots x_{n+1}^{\deg f_{n+1}} \\ &= \sum_{a_1, \ldots a_{n+1} \geq 0} c_{a_1, \ldots a_{n+1}}(q) x_1^{a_1} \cdots x_{n+1}^{a_{n+1}}.
\end{align}
Here, intuitively, $H(f_1, \ldots f_{n+1})$ is the weighted version of $\prod\limits_{i \sim j} \left(\frac{f_i}{f_j}\right)$. By definition, we have
\begin{equation}
c_{a_1, \ldots a_{n+1}}(q) = \sum_{\substack{f_1, \ldots f_{n+1} \in \mathbb{F}_q[t] \text{ monic} \\ \deg(f_i)=a_i}} H(f_1, \ldots f_{n+1}).
\end{equation}
When $q$ is fixed, we may simply write $c_{a_1, \ldots a_{n+1}}$. 

In the theory for finite Weyl groups, the $H(f_1, \ldots f_{n+1})$ can be chosen in an ad-hoc way, with the goal of making the functional equations true. We will see after Proposition (\ref{diagonalcoeffs}) below that this leads to a unique construction. For Kac-Moody Weyl groups, however, many possible choices for $H(f_1, \ldots f_{n+1})$ yield the same functional equations. Here we give an axiomatic definition of $H(f_1, \ldots f_{n+1})$. The four axioms are due to Diaconu and Pasol \cite{DP}, who study certain multiple Dirichlet series associated to moments of quadratic L-functions. Axioms (\ref{localglobal}) and (\ref{dominance}) are consequences of Poincar\'{e} duality on parameter spaces of hyperelliptic curves--the weights $H(f_1, \ldots f_{n+1})$ come from compactifying these spaces. Although the most general geometric context for these axioms is not yet completely understood, they seem to give the ``right'' definition for $Z(x_1, \ldots x_{n+1})$ in the classical and Kac-Moody cases: right in the sense of being optimal for computing moments of L-functions and character sums, and, we conjecture, appearing as Whittaker coefficients in metaplectic Eisenstein series.

The first axiom is familiar from the theory for classical Weyl groups:

\begin{axiom}[Twisted Multiplicativity] \label{twistedmult}
For $f_1\cdots f_{n+1}, g_1\cdots g_{n+1}\in \mathbb{F}_q[t]$ relatively prime, we have 
\begin{equation}
H(f_1 g_1, \ldots f_{n+1} g_{n+1})=H(f_1, \ldots f_{n+1})H(g_1, \ldots g_{n+1}) \prod_{i \sim j} \left(\frac{f_i}{g_j}\right).
\end{equation}
\end{axiom}

Thus it suffices to describe $H(p^{a_1}, \ldots p^{a_{n+1}})$ for $p \in \mathbb{F}_q[t]$ prime. The next two axioms give a characterization. They describe how the weights $H$ and coefficients $c$ vary as the underlying finite field $\F_q$ varies. They appear as axioms for the first time in the work of Diaconu and Pasol \cite{DP}, but can be proven as propositions in the theory for finite Weyl groups.

\begin{axiom}[Local to Global Principle]\label{localglobal}
The coefficients $c_{a_1, \ldots a_{n+1}}(q)$ and $H(p^{a_1}, \ldots p^{a_{n+1}})$ are polynomials in $q$ and $q^{\deg p}$ respectively, of degree $a_1+\cdots + a_{n+1}$. Furthermore, 
\begin{equation}
q^{(a_1+\cdots + a_{n+1})\deg p} c_{a_1, \ldots a_{n+1}}(q^{-\deg p})=H(p^{a_1}, \ldots p^{a_{n+1}}).
\end{equation}
\end{axiom}

Note that it only makes sense to evaluate $c_{a_1, \ldots a_{n+1}}$ at negative powers of $q$ after asserting that $c_{a_1, \ldots a_{n+1}}(q)$ is a polynomial.

\begin{axiom}[Dominance]\label{dominance}
The polynomial $H(p^{a_1}, \ldots p^{a_{n+1}})$ has degree less than $\frac{a_1+\cdots + a_{n+1}-1}{2}$. Equivalently, $c_{a_1, \ldots a_{n+1}}(q)$ has nonzero terms only in degrees greater than $\frac{a_1+\cdots + a_{n+1}+1}{2}$. The only exceptions are for $H(1,\ldots 1)$, $H(1, \ldots 1, p, 1, \ldots 1)$, $c_{0,\ldots 0}(q)$, and $c_{0, \ldots 0, 1, 0, \ldots 0}(q)$.
\end{axiom}

In concrete terms, the Dominance axiom states that the contribution of correction terms $H(p^{a_1}, \ldots p^{a_{n+1}})$ is as small as possible; the weighting affects the computation of moments of L-functions and character sums as little as possible. The final axiom is essentially a normalization assumption, and is of lesser importance:

\begin{axiom}[Initial Conditions]\label{initialconditions}
We have $H(1, \ldots 1, f_i, 1, \dots 1)=1$ for all $f_i \in \mathbb{F}_q[t]$, or equivalently, $c_{0, \ldots 0, a_i, 0, \ldots 0}(q)=q^{a_i}$.
\end{axiom}

Let $\x=(x_1, \ldots x_{n+1})$. The main theorem of this paper is as follows:

\begin{thm}\label{main}
Suppose that $\Gamma$ is the Dynkin diagram of a simply-laced affine Kac-Moody root system:

\includegraphics[scale=.30]{affine}

\noindent excepting $\tilde{A}_n$ for $n$ even. There exists a unique series $Z(\x)$ satisfying the four axioms. This series has meromorphic continuation to $|\x^{\alpha_0}|<q^{-\text{ht}(\alpha_0)}$, where $\alpha_0$ is the imaginary root of the root system $\Gamma$, and group of functional equations isomorphic to $W$. In the case of $\tilde{A}_n$ for $n$ odd, it has meromorphic continuation to $|\x^{\alpha_0}|<q^{-\text{ht}(\alpha_0)/2}$, which is the largest possible domain.
\end{thm}

We expect meromorphic continuation to $|\x^{\alpha_0}|<q^{-\text{ht}(\alpha_0)/2}$ to hold in all types, and give a specific conjecture which implies this. We also expect the same theorem to hold for $\tilde{A}_n$ with $n$ even. 

We will also give some evidence that the series constructed here are the correct ones for the Eisenstein conjecture: namely, that they have poles corresponding to imaginary roots in the root system.

In the case of finite Weyl groups, an analogous theorem holds. The axioms produce the same series as other known methods of construction; this series has meromorphic continuation to all of $\mathbb{C}^{n+1}$, and is in fact a rational function. The Eisenstein conjecture is known in many cases \cite{BBF1, FZ}.

In the case of non-affine Kac-Moody Weyl groups, such as those considered by Diaconu and Pasol, the axioms must be relaxed in a particular way: the coefficients $c_{a_1, \ldots a_{n+1}}(q)$ and $H(p^{a_1}, \ldots p^{a_{n+1}})$ may not be polynomials, though they consist of terms which have a well-defined notion of degree and can be evaluated at $q^{-1}$, for example, the Ramanujan tau function with degree $11/2$, or Fourier coefficients of cusp forms more generally. Diaconu and Pasol prove the uniqueness, though not necessarily the existence, of series satisfying the relaxed axioms for an infinite family of Kac-Moody groups.

\section{The Axioms Imply the Functional Equations}

For now, we set aside the more subtle questions of existence, uniqueness of $Z(\x)$, and study its basic properties under the assumption that it exists. Axioms (\ref{localglobal}) and (\ref{dominance}) imply that $Z(\x)$ converges absolutely in the domain $\lbrace (x_1, \ldots x_{n+1}) \in \C^{n+1} : \text{all } |x_i|<q^{-1} \rbrace$, and hence defines a holomorphic function in this domain. We will meromorphically continue to a larger domain below. For now, we verify that $Z(\x)$ has a group $W$ of functional equations. The crucial fact is the following:

\begin{prop} \label{axiomsimplyfes}
Suppose that we have a choice of weights $H(p^{a_1},\ldots p^{a_{n+1}})$ and coefficients $c_{a_1,\ldots a_{n+1}}(q)$ satisfying the axioms. Fix $a_1, \ldots a_{i-1}, a_{i+1}, \ldots a_{n+1}$, and let
\begin{align}
\lambda(x_i) &= \sum_{a_i=0}^{\infty} c_{a_1, \ldots a_i, \ldots a_{n+1}}(q) x_i^{a_i} \\
\lambda_p(x_i^{\deg p}) &= \sum_{a_i=0}^{\infty} H(p^{a_1}, \ldots p^{a_i}, \ldots p^{a_{n+1}}) x_i^{a_i \deg p}.
\end{align}
If $\sum\limits_{j \sim i} a_j$ is odd, then these series are polynomials of degree $\sum\limits_{j \sim i} a_j-1$, satisfying:
\begin{align}
&(q^{1/2} x_i)^{\sum\limits_{j \sim i} a_j-1} \lambda(q^{-1}x_i^{-1}) = \lambda(x_i) \label{oddfactor} \\ 
&(q^{1/2} x_i)^{(\sum\limits_{j \sim i} a_j-1)\deg p} \lambda_p(q^{-1}x_i^{-1}) = \lambda_p(x_i). \label{poddfactor}
\end{align}
If $\sum\limits_{j \sim i} a_j$ is even, then these series are rational functions, with denominators $1-qx_i$, $1-x_i^{\deg p}$ respectively and numerators of degree $\sum\limits_{j \sim i} a_j$, satisfying:
\begin{align}
&(q^{1/2} x_i)^{\sum\limits_{j \sim i} a_j}(1-x_i^{-1})\lambda(q^{-1}x_i^{-1}) = (1-qx_i)\lambda(x_i) \label{evenfactor} \\ 
&(q^{1/2} x_i)^{(\sum\limits_{j \sim i} a_j)\deg p}(1-q^{-\deg p}x_i^{-\deg p})\lambda_p(q^{-1}x_i^{-1}) = (1-x_i^{\deg p})\lambda_p(x_i). \label{pevenfactor}
\end{align}
\end{prop}

\begin{proof}
First, note that the statements for $\lambda$ and $\lambda_p$ are equivalent, by Axiom (\ref{localglobal}). Before proving these functional equations, we translate them into linear relations on the coefficients $c_{a_1,\ldots a_{n+1}}(q)$. If $\sum\limits_{j \sim i} a_j$ is odd, then (\ref{oddfactor}) implies:
\begin{equation}
c_{a_1,\ldots a_i, \ldots a_{n+1}}(q) = q^{a_i-(\sum\limits_{j \sim i} a_j-1)/2} c_{a_1,\ldots \sum\limits_{j \sim i} a_j-1-a_i, \ldots a_{n+1}}(q) \label{oddrecursion}
\end{equation}
and if $\sum\limits_{j \sim i} a_j$ is even, then (\ref{evenfactor}) implies:
\begin{align}
&c_{a_1,\ldots a_i, \ldots a_{n+1}}(q)-q c_{a_1,\ldots a_i-1, \ldots a_{n+1}}(q) \nonumber \\
& \ \ = q^{a_i-(\sum\limits_{j \sim i} a_j)/2} (c_{a_1,\ldots \sum\limits_{j \sim i} a_j-a_i, \ldots a_{n+1}}(q)-q c_{a_1,\ldots \sum\limits_{j \sim i} a_j-a_i-1, \ldots a_{n+1}}(q)). \label{evenrecursion}
\end{align}
In the following chapters, we will very often use the functional equations this way, as linear recurrences on the coefficients. Of course, it is also possible to write linear recurrences on the $H(p^{a_1},\ldots p^{a_{n+1}})$.

We proceed by induction on $\sum\limits_{j \neq i} a_j$. When $\sum\limits_{j \neq i} a_j=0$, the proposition follows from Axiom (\ref{initialconditions}). For the inductive step, fix $f_1, \ldots f_{i-1}, f_{i+1}, \ldots f_{n+1}$ of degrees  $a_1, \ldots a_{i-1}, a_{i+1}, \ldots a_{n+1}$, and consider 
\begin{align}
&L_{f_1, \ldots f_{i-1}, f_{i+1}, \ldots f_{n+1}}(x_i)=\sum_{f_i  \in \mathbb{F}_q[t] \text{ monic}} H(f_1, \ldots f_i, \ldots f_{n+1}) x_i^{\deg f_i} \nonumber \\
&=\prod_{\substack{j\sim k, j, k \neq i \\ p_j\neq p_k \\ p_j|f_j, \ p_k|f_k}} \left(\frac{p_j^{v_{p_j}(f_j)}}{p_k^{v_{p_k}(f_k)}}\right) \prod_p \left(\sum_{a_i=0}^{\infty} H(p^{v_p(f_1)},\ldots p^{a_i}, \ldots p^{v_p(f_{n+1})}) \prod_{\substack{j \sim i \\ p_j \neq p \\ p_j|f_j}} \left(\frac{p^{a_i}}{p_j^{v_{p_j}(f_j)}}\right) x_i^{a_i \deg p}\right)
\end{align}
where the products are over $p, p_j, p_k \in \mathbb{F}_q[t]$ prime, and $v_p(f)$ denotes the multiplicity of the prime factor $p$ in $f$. This Euler product formula follows from Axiom (\ref{twistedmult}). Furthermore, if we set $g$ as the squarefree part of $\prod\limits_{j \sim i} f_j$, then the Euler factors differ from those of the L-function
\begin{equation}
L(x_i, \chi_g)=\prod_p \left(1-\left(\frac{p}{g}\right) x_i^{\deg p}\right)^{-1}
\end{equation}
at finitely many places $p$, namely those dividing $\prod\limits_{j \neq i} f_j$. Let us consider how these modified factors contribute to the product. 

First, suppose $\sum\limits_{j \sim i} v_p(f_j)$ is odd. Then the modified Euler factor at $p$ is $\lambda_p(\pm x_i^{\deg p})$ (where the sign is determined by the residue of $p$ modulo the part of $g$ coprime to $p$) instead of $1$. If the inductive hypothesis applies to this factor, then it satisfies equation (\ref{poddfactor}) above. On the other hand, if $\sum\limits_{j \sim i} v_p(f_j)$ is even, then the modified Euler factor at $p$ is $\lambda_p(\pm x_i^{\deg p})$ instead of $(1 \mp x_i^{\deg p})^{-1}$. Hence the L-series is multiplied by $(1 \mp x_i^{\deg p})\lambda_p(\pm x_i^{\deg p})$, which satisfies equation (\ref{pevenfactor}) above if the inductive hypothesis applies.

If $\deg g$ is odd, then $L(x_i, \chi_g)$ is a polynomial of degree $\deg g -1$ satisfying
\begin{equation}
(q^{1/2} x_i)^{\deg g-1} L(q^{-1}x_i^{-1}, \chi_g) = L(x_i, \chi_g)
\end{equation} 
and if $\deg g$ is even, then $L(x_i, \chi_g)$ is a polynomial of degree $\deg g-1$ (or if $g=1$, it is the zeta function $(1-qx_i)^{-1}$), satisfying
\begin{equation}
(q^{1/2} x_i)^{\deg g}(1-x_i^{-1})L(q^{-1}x_i^{-1}, \chi_g) = (1-qx_i)L(x_i, \chi_g).
\end{equation}
The modified factors add $\sum\limits_{j \sim i} v_p(f_j)-1$ if $\sum\limits_{j \sim i} v_p(f_j)$ is odd, or $\sum\limits_{j \sim i} v_p(f_j)$ if $\sum\limits_{j \sim i} v_p(f_j)$ is even, to the exponent of $q^{1/2}x_i$. We conclude that, if $\sum\limits_{j \sim i} a_j$ is odd, then $L_{f_1, \ldots f_{i-1}, f_{i+1}, \ldots f_{n+1}}(x_i)$ is a polynomial of degree $\sum\limits_{j \sim i} a_j-1$ satisfying the functional equation (\ref{oddfactor}). If $\sum\limits_{j \sim i} a_j$ is even, then $L_{f_1, \ldots f_{i-1}, f_{i+1}, \ldots f_{n+1}}(x_i)$ is a polynomial of degree $\sum\limits_{j \sim i} a_j-1$ (or a rational function with denominator $1-qx_i$ and numerator of degree $\sum\limits_{j \sim i} a_j$) satisfying the functional equation (\ref{evenfactor}).

We may follow this reasoning as long as the inductive hypothesis applies to all the modified Euler factors. That is, we must have $\sum\limits_{j \neq i} v_p(f_j) < \sum\limits_{j \neq i} a_j$. This occurs in all cases except when $p$ is linear and each $f_j=p^{a_j}$. We write
\begin{align}
\lambda(x_i) &= \sum_{\substack{f_1, \ldots f_{i-1}, f_{i+1}, \ldots f_{n+1} \\ \deg f_j=a_j}} L_{f_1, \ldots f_{i-1}, f_{i+1}, \ldots f_{n+1}}(x_i) \nonumber \\
&= \sum_{p \in \mathbb{F}_q[t] \text{ linear}} L_{p^{a_1}, \ldots p^{a_{i-1}}, p^{a_{i+1}}, \ldots p^{a_{n+1}}}(x_i) \nonumber \\
& \ + \sum_{\substack{(f_1, \ldots f_{i-1}, f_{i+1}, \ldots f_{n+1}) \\ \neq(p^{a_1}, \ldots p^{a_{i-1}}, p^{a_{i+1}}, \ldots p^{a_{n+1}})}} L_{f_1, \ldots f_{i-1}, f_{i+1}, \ldots f_{n+1}}(x_i).
\end{align}
If $\sum\limits_{j \sim i} a_j$ is odd, then $L_{p^{a_1}, \ldots p^{a_{i-1}}, p^{a_{i+1}}, \ldots p^{a_{n+1}}}(x_i)=\lambda_p(x_i)$. Hence
\begin{equation}
\lambda(x_i)-q \lambda_p(x_i)
\end{equation} 
satisfies the functional equation (\ref{oddfactor}). If $\sum\limits_{j \sim i} a_j$ is even, then $L_{p^{a_1}, \ldots p^{a_{i-1}}, p^{a_{i+1}}, \ldots p^{a_{n+1}}}(x_i)=\frac{1-x_i}{1-qx_i}\lambda_p(x_i)$, and
\begin{equation}
\lambda(x_i)-q \left(\frac{1-x_i}{1-qx_i}\right)\lambda_p(x_i)
\end{equation}
satisfies the functional equation (\ref{evenfactor}). Let us translate these functional equations into coefficient relations. If $\sum\limits_{j \sim i} a_j$ is odd, then 
\begin{align}
&c_{a_1,\ldots a_i, \ldots a_{n+1}}(q) -qH(p^{a_1}, \ldots p^{a_i}, \ldots p^{a_{n+1}}) \nonumber \\ 
&= q^{a_i-(\sum\limits_{j \sim i} a_j-1)/2} (c_{a_1,\ldots \sum\limits_{j \sim i} a_j-1-a_i, \ldots a_{n+1}}(q)-q(H(p^{a_1}, \ldots p^{\sum\limits_{j \sim i} a_j-1-a_i}, \ldots p^{a_{n+1}}))
\end{align}
By Axiom (\ref{dominance}), comparing coefficients with degree greater than $\frac{a_1+\cdots + a_{n+1}+1}{2}$ in $q$, we recover equation (\ref{oddrecursion}). If $\sum\limits_{j \sim i} a_j$ is even, then
\begin{align}
&c_{a_1,\ldots a_i, \ldots a_{n+1}}(q)-q c_{a_1,\ldots a_i-1, \ldots a_{n+1}}(q) \nonumber \\
& \ \ -qH(p^{a_1},\ldots p^{a_i}, \ldots p^{a_{n+1}}) +qH(p^{a_1},\ldots p^{a_i-1}, \ldots p^{a_{n+1}}) \nonumber \\
& = q^{a_i-(\sum\limits_{j \sim i} a_j)/2} \left(c_{a_1,\ldots \sum\limits_{j \sim i} a_j-a_i, \ldots a_{n+1}}(q)-q c_{a_1,\ldots \sum\limits_{j \sim i} a_j-a_i-1, \ldots a_{n+1}}(q)\right. \nonumber \\
& \ \ \left. -qH(p^{a_1},\ldots p^{\sum\limits_{j \sim i} a_j-a_i}, \ldots p^{a_{n+1}})+qH(p^{a_1},\ldots p^{\sum\limits_{j \sim i} a_j-a_i-1}, \ldots p^{a_{n+1}})\right)
\end{align}
and again, by Axiom (\ref{dominance}) we recover equation (\ref{evenrecursion}).
\end{proof}

The functional equations satisfied by the full series $Z(\x)$ may be modeled as follows. Let $\sigma_i(\x)$ be defined by
\begin{equation}
(\sigma_i(\x))_j=\left\lbrace\begin{array}{cc} q^{-1}x_j^{-1} & \text{ if } j=i \\ q^{1/2}x_i x_j & \text{ if } j \sim i \\ x_j & \text{ otherwise} \end{array} \right. . 
\end{equation}
One can check directly that $\sigma_i^2=1$, $\sigma_i \sigma_j \sigma_i=\sigma_j \sigma_i \sigma_j$ if $j \sim i$, and $\sigma_i \sigma_j= \sigma_j \sigma_i$ if $j\nsim i$, to show that the $\sigma_i$ satisfy the defining relations of simple reflections generating the Weyl or Coxeter group $W$.

Let $Z_{\sum\limits_{j\sim i} a_j \text{ odd}}(\x)$, $Z_{\sum\limits_{j\sim i} a_j \text{ even}}(\x)$ denote the power series $Z$ restricted to terms where $\sum\limits_{j\sim i} a_j$ is odd or even respectively. We have the following functional equations:
\begin{align}
& Z_{\sum\limits_{j\sim i} a_j \text{ odd}}(\sigma_i(\x))=q^{1/2} x_i Z_{\sum\limits_{j\sim i} a_j \text{ odd}}(\x), \\
& (1-x_i^{-1}) Z_{\sum\limits_{j\sim i} a_j \text{ even}}(\sigma_i(\x))= (1-qx_i) Z_{\sum\limits_{j\sim i} a_j \text{ even}}(\x).
\end{align}
At present, these functional equations are identities of formal power series only. Their meaning is simply that the coefficient relations (\ref{oddrecursion}) and (\ref{evenrecursion}) hold. In order to interpret the functional equations as equalities of meromorphic functions, $Z$ must have meromorphic continuation to neighborhoods of both $\x$ and $\sigma_i(\x)$, which we have not yet shown. 

If we define the ``$p$-part'' of the multiple Dirichlet series, for $p \in \mathbb{F}_q[t]$ prime, as 
\begin{equation}
Z_p(\x) = \sum_{a_1, \ldots a_{n+1} \geq 0} H(p^{a_1}, \ldots p^{a_{n+1}}) x_1^{a_1 \deg p} \cdots x_{n+1}^{a_{n+1} \deg p},
\end{equation} 
then we have the equivalent local functional equations:
\begin{align}
& Z_{p, \sum\limits_{j\sim i} a_j \text{ odd}}(\sigma_i(\x))=q^{\deg p/2} x_i^{\deg p} Z_{p, \sum\limits_{j\sim i} a_j \text{ odd}}(\x), \label{oddlocalfe} \\
& (1-q^{-\deg p}x_i^{-\deg p})Z_{p, \sum\limits_{j\sim i} a_j \text{ even}}(\sigma_i(\x)) = (1-x_i^{\deg p})Z_{p, \sum\limits_{j\sim i} a_j \text{ even}}(\x). \label{evenlocalfe}
\end{align}

%Title/short description of chapter
\chapter{Consequences of the Functional Equations}\label{chap:3}
\section{The Family of Series Satisfying the Functional Equations}

In this chapter $\Gamma$ will be the Dynkin diagram of a classical simply-laced root system:

\includegraphics[scale=.30]{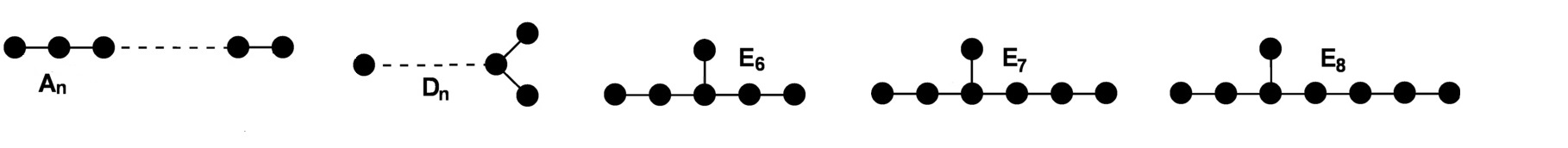}

\noindent or a simply-laced affine Kac-Moody root system:

\includegraphics[scale=.30]{affine}

\noindent with Weyl group $W$. We have shown that a multiple Dirichlet series $Z(x_1, \ldots x_{n+1})$ satisfying the four axioms must have a group of functional equations isomorphic to $W$, generated by the following simple reflections for $i=1, \ldots n+1$:
\begin{align}
& Z_{\sum\limits_{j\sim i} a_j \text{ odd}}(\sigma_i(\x))=q^{1/2} x_i Z_{\sum\limits_{j\sim i} a_j \text{ odd}}(\x), \label{oddglobalfe} \\
& (1-x_i^{-1}) Z_{\sum\limits_{j\sim i} a_j \text{ even}}(\sigma_i(\x))= (1-qx_i) Z_{\sum\limits_{j\sim i} a_j \text{ even}}(\x) \label{evenglobalfe}
\end{align}
where
\begin{equation}
(\sigma_i(\x))_j=\Bigg{\lbrace}\begin{array}{cc} q^{-1}x_j^{-1} & \text{ if } j=i \\ q^{1/2}x_i x_j & \text{ if } j \sim i \\ x_j & \text{ otherwise} \end{array}. 
\end{equation}
The functional equations make sense as relations of formal power series regardless of the analytic behavior of $Z$. In the following propositions, we do not assume the four axioms--we only assume the functional equations and study their consequences. First we will prove a result describing the family of all power series satisfying the functional equations.
\begin{prop} \label{diagonalcoeffs}
Suppose that a power series 
\begin{equation}
Z(x_1, \ldots x_{n+1})=\sum_{a_1, \ldots a_{n+1} \geq 0} c_{a_1, \ldots a_{n+1}} x_1^{a_1} \cdots x_{n+1}^{a_{n+1}}
\end{equation}
has group of functional equations $W$ generated by (\ref{oddglobalfe}) and (\ref{evenglobalfe}). If $W$ is a finite Weyl group, then $Z$ is unique up to a constant multiple. If $W$ is an affine Weyl group, then $Z$ is uniquely determined by the one-parameter family of ``diagonal'' coefficients $c_{m \alpha_0}$ where $\alpha_0$ is the imaginary root and $m \in \mathbb{Z}_{\geq 0}$.
\end{prop}
This theorem is well-known in the classical case, though it is not usually stated this way in the literature. A proof for the affine Weyl group $\tilde{D}_4$ is given in Bucur-Diaconu \cite{BD}.

Next, we prove the existence of a power series satisfying the functional equations with the correct analytic behavior.
\begin{prop} \label{averaging}
There exists a power series $Z(x_1, \ldots x_{n+1})$ with group of functional equations $W$ generated by (\ref{oddglobalfe}) and (\ref{evenglobalfe}), which has meromorphic continuation to $\mathbb{C}^{n+1}$ if $W$ is finite or $|\x^{\alpha_0}|<q^{-\text{ht}(\alpha_0)/2}$ if $W$ is affine. Moreover, we have an explicit denominator $D(x_1, \ldots x_{n+1})$ such that $D(x_1, \ldots x_{n+1})Z(x_1, \ldots x_{n+1})$ is holomorphic in these domains. 
\end{prop}
This is proven by the averaging construction of Chinta-Gunnells \cite{CG1, CG2}. A proof for symmetrizable Kac-Moody root systems is given in Lee-Zhang \cite{LZ}, but their formalism is very different, so we reprove it here. The averaging construction does not give a series which satisfies the four axioms, but we will modify it in the following chapters. 

\begin{proof}[Proof of Proposition (\ref{diagonalcoeffs})]
The proof is based on a detailed examination of the coefficient relations from Proposition (\ref{axiomsimplyfes}), which we restate here. For $\sum\limits_{j \sim i} a_j$ odd:
\begin{equation}
c_{a_1,\ldots a_i, \ldots a_{n+1}} = q^{a_i-(\sum\limits_{j \sim i} a_j-1)/2} c_{a_1,\ldots \sum\limits_{j \sim i} a_j-1-a_i, \ldots a_{n+1}} \label{oddrecursion2}
\end{equation}
and for $\sum\limits_{j \sim i} a_j$ even:
\begin{equation}
c_{a_1,\ldots a_i, \ldots a_{n+1}}-q c_{a_1,\ldots a_i-1, \ldots a_{n+1}} = q^{a_i-(\sum\limits_{j \sim i} a_j)/2} (c_{a_1,\ldots \sum\limits_{j \sim i} a_j-a_i, \ldots a_{n+1}}-q c_{a_1,\ldots \sum\limits_{j \sim i} a_j-a_i-1, \ldots a_{n+1}}). \label{evenrecursion2}
\end{equation}
For any list of indices $(a_1, \ldots a_{n+1})$, if we can find an $i$ such that $a_i>\frac{1}{2}\sum\limits_{j \sim i} a_j$, then the coefficient $c_{a_1, \ldots a_{n+1}}$ can be rewritten, via the $\sigma_i$ coefficient relation, in terms of coefficients with smaller $i$th index. For classical root systems, any coefficient $c_{a_1, \ldots a_{n+1}}$ can be reduced in this way, until we reach $c_{0, \ldots 0}$ (we assume that if any $a_i$ is negative, the coefficient is $0$). Thus in the classical case, the series is uniquely determined by $c_{0, \ldots 0}$--that is, it is determined up to a constant multiple. 

In the affine case, any nondiagonal coefficient can be reduced, but the diagonal coefficients $c_{m \alpha_0}$ cannot. The vector of indices $m\alpha_0$ is invariant under the Weyl group $W$. It follows immediately that all coefficients of the series are determined by the diagonal coefficients. Furthermore, the monomial $\x^{\alpha_0}$ is invariant under all $\sigma_i$, so multiplying $Z(x_1, \ldots x_{n+1})$ by any power series in $\x^{\alpha_0}$ does not affect the functional equations. In this way, we may obtain any possible family of diagonal coefficients. Hence the series $Z(x_1, \ldots x_{n+1})$ is completely determined by any choice of diagonal coefficients.
\end{proof}

This simple argument has many important consequences. 

\begin{cor} \label{onevariable}
In the affine case, the series $Z(x_1, \ldots x_{n+1})$ is determined by its functional equations up to multiplication by a power series in the the invariant monomial $\x^{\alpha_0}$.
\end{cor}

\begin{cor} \label{diagonaldominance}
If the undetermined coefficients $c_{m\alpha_0}$ (or just $c_{0, \ldots 0}$ in the finite case) satisfy Axiom (\ref{dominance}), then so do all the coefficients.
\end{cor}
\begin{proof}
Suppose that we have $c_{a_1, \ldots a_{n+1}}$ with $a_i>\frac{\sum\limits_{j \sim i} a_j}{2}$, and we apply the coefficient recurrence associated to $\sigma_i$. In the relations (\ref{oddrecursion2}) and (\ref{evenrecursion2}), every term $c_{a_1, \ldots a_i', \ldots a_{n+1}}$ is multiplied by a power of $q$ between $\frac{a_i-a_i'}{2}$ and $a_i-a_i'$. Hence if all the lower coefficients satisfy the Dominance Axiom, then $c_{a_1, \ldots a_{n+1}}$ does as well. Repeating this process recursively, we see that dominance for the undetermined coefficients implies dominance for all coefficients.
\end{proof}

Now we will show in the finite Weyl group case that the unique series $Z(\x)$ with $c_{0, \ldots 0}=1$ also satisfies the four axioms. We begin by constructing the local coefficients $H(f_1, \ldots f_{n+1})$. Let $H(1, \ldots 1)=1$. The local functional equations (\ref{oddlocalfe}) and (\ref{evenlocalfe}) can be used, exactly as in Proposition (\ref{diagonalcoeffs}), to construct all $H(p^{a_1}, \ldots p^{a_{n+1}})$, and then Axiom (\ref{twistedmult}) gives all $H(f_1, \ldots f_{n+1})$. Set
\begin{align}
Z(x_1, \ldots x_{n+1}) &= \sum_{f_1, \ldots f_{n+1} \in \mathbb{F}_q[t] \text{ monic}} H(f_1, \ldots f_{n+1}) x_1^{\deg f_1} \cdots x_{n+1}^{\deg f_{n+1}} \\ &= \sum_{a_1, \ldots a_{n+1} \geq 0} c_{a_1, \ldots a_{n+1}} x_1^{a_1} \cdots x_{n+1}^{a_{n+1}}.
\end{align}
It follows as in the proof of Proposition (\ref{axiomsimplyfes}) that $Z(x_1, \ldots x_{n+1})$ satisfies (\ref{oddglobalfe}) and (\ref{evenglobalfe}). Hence it is the unique such series with $c_{0, \ldots 0}=1$. Because the $c_{a_1, \ldots a_{n+1}}$ are all obtained recursively from $c_{0, \ldots 0}$ and the $H(p^{a_1}, \ldots p^{a_{n+1}})$ are all obtained from $H(1, \ldots 1)$ by analogous recurrences, Axiom (\ref{localglobal}) is satisfied. By Corollary (\ref{diagonaldominance}), Axiom (\ref{dominance}) is satisfied. 

One could try to proceed this way in the affine case; the difficulty is that, if we make an arbitrary choice of diagonal local coefficients $H(p^{m\alpha_0})$, the resulting diagonal global coefficients $c_{m\alpha_0}$ will not necessarily satisfy Axiom (\ref{localglobal}). We will give a different interpretation of Axiom (\ref{localglobal}) in the following section to resolve this problem.

\section{The Chinta-Gunnells Averaged Series}

Before proving Proposition (\ref{averaging}), we introduce some additional notation and explain why the domains of meromorphic continuation given in the proposition are optimal. If $\Phi$ is the root system corresponding to $\Gamma$, we let 
\begin{equation}
D(x_1, \ldots x_{n+1})=\prod_{\alpha \in \Phi^+} (1-q^{\text{ht}(\alpha)+1}\x^{2\alpha}).
\end{equation}
This is a finite product over positive roots if $W$ is finite, and an infinite product if $W$ is affine. Each factor corresponds to a polar divisor of $Z(x_1, \ldots x_{n+1})$--it follows from Axiom (\ref{initialconditions}) or from the functional equations that $Z$ has poles at $x_i=q^{-1}$ for all $i$, and translating these poles by the group $W$ gives a pole at $\x^{\alpha}=q^{-(\text{ht}(\alpha)+1)/2}$ for every $\alpha \in \Phi^+$. 

If $W$ is finite, then $D(x_1, \ldots x_{n+1})$ is meromorphic in $\mathbb{C}^{n+1}$. If $W$ is affine, then
\begin{equation}
D(x_1, \ldots x_{n+1})=\prod_{\substack{\alpha \in \Phi^+ \\ \alpha \leq \alpha_0}} \prod_{m=0}^{\infty} (1-q^{\text{ht}(\alpha)+m\text{ht}(\alpha_0)+1}\x^{2\alpha+2m\alpha_0}).
\end{equation}
Let $M(\x)=\text{Max}_{\alpha \in \Phi^+, \alpha \leq \alpha_0}|q^{\text{ht}(\alpha)+1}\x^{2\alpha}|$. Then the absolute value of the product above is bounded by:
\begin{equation}
\prod_{m=0}^{\infty} (1+M(\x)q^{m \text{ht}(\alpha_0)}|\x^{2m \alpha_0}|)^{\#\lbrace \alpha \in \Phi^+, \alpha \leq \alpha_0 \rbrace}.
\end{equation} 
If $|\x^{\alpha_0}|<q^{-\text{ht}(\alpha_0)/2}$, we can ignore the finitely many terms where $q^{m \text{ht}(\alpha_0)}|\x^{2m \alpha_0}| \geq M(\x)^{-1}$, and the logarithm of the remaining terms is bounded by a convergent geometric series. Thus $D(x_1, \ldots x_{n+1})$ converges absolutely to a holomorphic function in this domain. 

The transformation $\sigma_i$ permutes the factors of $D$, except for $(1-q^2x_i^2)$, which becomes $(1-x_i^{-2})$, so we have 
\begin{equation}
D(\sigma_i(x_1, \ldots x_{n+1}))=\left(\frac{1-x_i^{-2}}{1-q^2x_i^2}\right) D(x_1, \ldots x_{n+1})
\end{equation}
or, more generally for any $w\in W$,
\begin{equation}
D(w(x_1, \ldots x_{n+1}))=\left(\prod_{\alpha \in \Phi(w)} \frac{1-q^{-\text{ht}(\alpha)+1}\x^{-2\alpha}}{1-q^{\text{ht}(\alpha)+1}\x^{2\alpha}}\right) D(x_1, \ldots x_{n+1})
\end{equation}
where $\Phi(w)$ denotes $\Phi^+ \cap w^{-1}(\Phi^-)$.

In the affine case, the poles of $Z(x_1, \ldots x_{n+1})$, or zeroes of $D(x_1, \ldots x_{n+1})$, accumulate along the divisor $\x^{\alpha_0}=q^{-\text{ht}(\alpha_0)/2}$, the boundary of the Tits cone for $W$. $Z(x_1, \ldots x_{n+1})$ has an essential singularity at this divisor, so meromorphic continuation beyond it is impossible. Figure \ref{poles} shows the zeroes of $D(x_1, x_2, x_3, x_4)$ for the $\tilde{A}_3$ multiple Dirichlet series, after making the substitution $x_1, x_3 \mapsto q^{-s}$, $x_2, x_4 \mapsto q^{-t}$. The zeroes appear as lines in the real $st$ plane, accumulating at $s+t=1$. 
\\

\begin{figure}
\includegraphics[scale=.42]{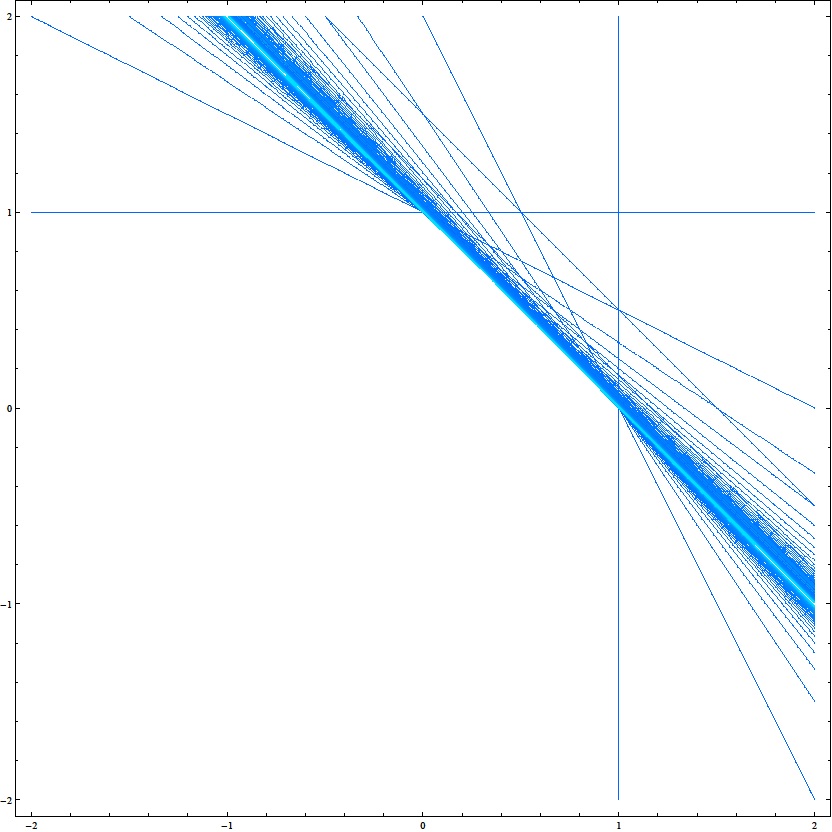}
\caption{Poles of the $\tilde{A}_3$ Multiple Dirichlet Series}
\label{poles}
\end{figure}

$D(x_1, \ldots x_{n+1})$ can be viewed as a deformation of the denominator in the Weyl character formula for $W$. We will also need a version of this denominator without the deformation: 
\begin{equation}
\Delta(x_1, \ldots x_{n+1})=\prod_{\alpha \in \Phi^+} (1-q^{\text{ht}(\alpha)}\x^{2\alpha}).
\end{equation}
This converges in the same domain as $D(x_1, \ldots x_{n+1})$, and satisfies
\begin{equation} \label{deltafe}
\Delta(w(x_1, \ldots x_{n+1}))=\left(\prod_{\alpha \in \Phi(w)} -q^{-\text{ht}(\alpha)}\x^{-2\alpha}\right) \Delta(x_1, \ldots x_{n+1}). 
\end{equation}

\begin{proof}[Proof of Proposition (\ref{averaging})]
This proof is modeled upon work of Chinta and Gunnells \cite{CG1}, who describe an action of finite Weyl groups on rational functions. The generalization to affine Weyl groups acting on power series is straightforward. We define maps $\epsilon_i:\mathbb{C}^{n+1}\to \mathbb{C}^{n+1}$ by  
\begin{equation}
(\epsilon_i(x_1, \ldots x_{n+1}))_j=\bigg{\lbrace}\begin{array}{cc} -x_j & \text{ if } j \sim i \\ x_j & \text{ else} \end{array}
\end{equation}
so that
\begin{align}
Z_{\sum\limits_{j\sim i} a_j \text{ odd}}(x_1, \ldots x_{n+1})&=\frac{Z(x_1, \ldots x_{n+1})-Z(\epsilon_i(x_1, \ldots x_{n+1}))}{2} \\
Z_{\sum\limits_{j\sim i} a_j \text{ even}}(x_1, \ldots x_{n+1})&=\frac{Z(x_1, \ldots x_{n+1})+Z(\epsilon_i(x_1, \ldots x_{n+1}))}{2}
\end{align}
For an arbitrary power series $f(\x)\in \mathbb{C}[[x_1, \ldots x_{n+1}]][x_1^{-1}, \ldots x_{n+1}^{-1}]$, we define 
\begin{equation}
(f|\sigma_i)(\x)=(1-x_i^{-1})(1-qx_i)^{-1} f_{\sum\limits_{j\sim i} a_j \text{ even}}(\sigma_i(\x))+q^{-1/2}x_i^{-1} f_{\sum\limits_{j\sim i} a_j \text{ odd}}(\sigma_i(\x)).
\end{equation}
The $(1-qx_i)^{-1}$ in this formula is shorthand for its geometric series expansion. Chinta and Gunnells prove that this definition extends to a $\mathbb{C}$-linear group action of $W$ on rational functions \cite[Lemma 3.2]{CG1}; the proof carries over to power series without modification. Note that our definition is slightly different from equation (3.13) of \cite{CG1}, because we are modeling the global, rather than local, functional equations. Indeed, equations (\ref{oddglobalfe}) and (\ref{evenglobalfe}) are equivalent to the statement that $(Z|w)=Z$ for all $w\in W$. 

To rephrase, we have
\begin{equation}
(f|\sigma_i)(\x)=\frac{x_i^{-1}}{2}\left(\frac{x_i-1}{1-qx_i}+q^{-1/2}\right) f(\sigma_i(\x))+\frac{x_i^{-1}}{2}\left(\frac{x_i-1}{1-qx_i}-q^{-1/2}\right)f(\epsilon_i\sigma_i(\x)).
\end{equation}
Let us set $K_i(\x, \delta)=\frac{x_i^{-1}}{2}\left(\frac{x_i-1}{1-qx_i}+(-1)^{\delta}q^{-1/2}\right)$. Then if $w\in W$ can be expressed as $\sigma_{i_1}\cdots \sigma_{i_{\ell}}$, we have
\begin{align}
(f|w)(\x)=\sum_{\delta_{i_1}, \ldots \delta_{i_{\ell}} \in \lbrace 0,1\rbrace} & K_{i_{\ell}}(\x, \delta_{i_{\ell}}) K_{i_{\ell-1}}(\epsilon_{i_{\ell}}^{\delta_{i_{\ell}}}\sigma_{i_{\ell}}(\x), \delta_{i_{\ell-1}}) K_{i_{\ell-2}}(\epsilon_{i_{\ell-1}}^{\delta_{i_{\ell-1}}}\sigma_{i_{\ell-1}}\epsilon_{i_{\ell}}^{\delta_{i_{\ell}}}\sigma_{i_{\ell}}(\x), \delta_{i_{\ell-2}}) \cdots \nonumber \\ &\cdots K_{i_1}(\epsilon_{i_2}^{\delta_{i_2}}\sigma_{i_2}\cdots\epsilon_{i_{\ell}}^{\delta_{i_{\ell}}}\sigma_{i_{\ell}}(\x), \delta_{i_1})f(\epsilon_{i_1}^{\delta_{i_1}}\sigma_{i_1}\cdots\epsilon_{i_{\ell}}^{\delta_{i_{\ell}}}\sigma_{i_{\ell}}(\x))
\end{align}
In particular, the monomials appearing in the $K_i$ are $x_{i_\ell}$, $(\sigma_{i_{\ell}}(\x))_{i_{\ell-1}}$, $(\sigma_{i_{\ell-1}}\sigma_{i_{\ell}}(\x))_{i_{\ell-2}}$, ... $(\sigma_{i_2}\cdots\sigma_{i_{\ell}}(\x))_{i_1}$. If $w=\sigma_{i_1}\cdots \sigma_{i_{\ell}}$ is an expression in reduced form, then these are precisely $q^{(\text{ht}(\alpha)-1)/2}\x^{\alpha}$ for $\alpha \in \Phi(w)$. 

We will show that $Z(\x):=\sum_{w \in W} \Delta(w(\x))^{-1}(1|w)(\x)$ is a power series satisfying functional equations (\ref{oddglobalfe}) and (\ref{evenglobalfe}), with meromorphic continuation to all $\x$ if $W$ is finite and to all $|\x^{\alpha_0}|<q^{-\text{ht}(\alpha_0)/2}$ if $W$ is affine. First note that, by equation (\ref{deltafe}), $\Delta(w(\x))^{-1}=\left(\prod_{\alpha \in \Phi(w)} q^{\text{ht}(\alpha)}\x^{2 \alpha}\right)\Delta(\x)^{-1}$, and $\Delta(\x)^{-1}$ is a power series in $\mathbb{C}[[\x]]$. On the other hand, $(1|w)(\x)$ is $\prod_{\alpha \in \Phi(w)} q^{(1-\text{ht}(\alpha))/2}\x^{-\alpha}$ times a power series in $\mathbb{C}[[\x]]$. Thus the summand at $w$, $\Delta(w(\x))^{-1}(1|w)(\x)$ is a power series divisible by $q^{(\ell(w)+\sum_{\alpha \in \Phi(w)} \text{ht}(\alpha))/2}\x^{\sum_{\alpha \in \Phi(w)} \alpha}$. In particular, no negative exponents of $x_i$ appear in $Z(\x)$. Moreover, in the affine case, since $\sum_{\alpha \in \Phi(w)} \alpha \to \infty$ as $\ell(w) \to \infty$, only finitely many terms in the sum over $w$ contribute to each coefficient of $Z(\x)$, so the sum is a well-defined power series. 

To show that $Z(\x)$ satisfies the functional equations, we use a simple lemma, also proven in \cite{CG1}:
\begin{lemma}
If $f(\x), g(\x)$ are power series with $f(\x)$ even in all $x_i$, then 
\begin{equation}
(fg|w)(\x)=f(w(\x))(g|w)(\x).
\end{equation}
\end{lemma} 
Then, since $\Delta$ is even, 
\begin{equation}
(Z|w')(\x) = \sum_{w \in W} \Delta(ww'(\x))^{-1}(1|ww')(\x)=Z(\x)
\end{equation}
This argument applies whether we consider $Z$ as a formal power series, or a meromorphic function on a suitable domain.

Next we show that $D(\x)\Delta(\x)Z(\x)$ converges absolutely in the desired domain. If $W$ is finite, this is automatic, as $D(\x)\Delta(\x)Z(\x)$ is a polynomial. If $W$ is affine, recall that we used the bound 
\begin{equation}
|1-q^{\text{ht}(\alpha)+m\text{ht}(\alpha_0)+1}\x^{2\alpha+2 m \alpha_0}| \leq 1+M(\x)q^{m\text{ht}(\alpha_0)}|\x^{2m\alpha_0}|
\end{equation}
for factors of $D(\x)$. Furthermore $(1+M(\x)q^{m\text{ht}(\alpha_0)}|\x^{2m\alpha_0}|)(1+q^{-1/2})$ bounds 
\begin{equation}
\left|\left(1-q^{\text{ht}(\alpha)+m\text{ht}(\alpha_0)+1}\x^{2\alpha+2 m \alpha_0}\right)\left(\frac{\pm q^{(\text{ht}(\alpha)+m(\text{ht}(\alpha_0)-1)/2}\x^{\alpha+m\alpha_0}-1}{1 \pm q^{(\text{ht}(\alpha)+m\text{ht}(\alpha_0)+1)/2}\x^{\alpha+m\alpha_0}} \pm q^{-1/2}\right)\right|.
\end{equation}
Combining these bounds, and using equation (\ref{deltafe}), we have:
\begin{align}
|D(\x)\Delta(\x)Z(\x)| \leq & \left(\sum_{w \in W} (1+q^{-1/2})^{\ell(w)} q^{(\ell(w)+\sum_{\alpha \in \Phi(w)} \text{ht}(\alpha))/2}|\x^{\sum_{\alpha \in \Phi(w)} \alpha}|\right) \nonumber \\ &\left(\prod_{m=0}^{\infty} (1+M(\x)q^{m \text{ht}(\alpha_0)}|\x^{2m \alpha_0}|)^{\#\lbrace \alpha \in \Phi^+, \alpha \leq \alpha_0 \rbrace}\right)
\end{align}
The second product is independent of $w$ and converges absolutely for $|\x^{\alpha_0}|<q^{-\text{ht}(\alpha_0)/2}$, so it suffices to show the convergence of the first sum. Let $A=\#\lbrace \alpha \in \Phi^+, \alpha \leq \alpha_0 \rbrace$. Assuming that the roots in $\Phi(w)$ are as small as possible, we still must have 
\begin{equation}
\sum_{\alpha \in \Phi(w)} \alpha \geq \frac{1}{2} A \bigg{\lfloor}\frac{\ell(w)}{A}\bigg{\rfloor} \left(\bigg{\lfloor}\frac{\ell(w)}{A}\bigg{\rfloor} -1\right) \alpha_0,
\end{equation}
which grows as $\frac{1}{2}\ell(w)^2 + O(\ell(w))$. We have 
\begin{align}
&\sum_{w \in W} (1+q^{-1/2})^{\ell(w)} q^{(\ell(w)+\sum_{\alpha \in \Phi(w)} \text{ht}(\alpha))/2}|\x^{\sum_{\alpha \in \Phi(w)} \alpha}| \nonumber \\
&\leq \sum_{w \in W} (q^{1/2}+1)^{\ell(w)}M(\x)^{\ell(w)}|q^{\text{ht}(\alpha_0)/2}\x^{\alpha_0}|^{\frac{1}{2}\ell(w)^2+O(\ell(w))} \nonumber \\ 
&\leq \sum_{\ell=0}^{\infty} (n+1)^{\ell}(q^{1/2}+1)^{\ell}M(\x)^{\ell}|q^{\text{ht}(\alpha_0)/2}\x^{\alpha_0}|^{\frac{1}{2}\ell^2+O(\ell)}
\end{align}
where the last inequality follows from the weak bound $\#\lbrace w\in W: \ell(w)=\ell \rbrace \leq (n+1)^{\ell}$. Hence the sum converges for $|\x^{\alpha_0}|<q^{-\text{ht}(\alpha_0)/2}$. 

Finally we show that $D(\x)Z(\x)$ is also holomorphic, by proving that $D(\x)\Delta(\x)Z(\x)$ vanishes at all the zeroes of $\Delta(\x)$, i.e. at $\x^\alpha= \pm q^{-\text{ht}(\alpha)/2}$ for all $\alpha \in \Phi^+$. For any root $\alpha$ we can find $w \in W$ mapping $\alpha$ to the simple root $e_i$, and the $w$-functional equation allows us to express $D(\x)\Delta(\x)Z(\x)$ with $\x^\alpha= \pm q^{-\text{ht}(\alpha)/2}$ in terms of various $D(\x)\Delta(\x)Z(\x')$ with $x_i'=\pm q^{1/2}$. Thus it suffices to show that $D(\x)\Delta(\x)Z(\x)$ vanishes at all $\x$ with $x_i=\pm q^{1/2}$. 

We will decompose $W$ into two parts: the set $\lbrace w : e_i \in \Phi(w) \rbrace$, or elements with a reduced expression ending in $\sigma_i$, and its complement. Multiplication by $\sigma_i$ gives a bijection between these sets. We may write $D(\x)\Delta(\x)Z(\x)$ as
\begin{equation}
D(\x) \left( \sum_{\substack{w \in W \\ e_i \not{\in} \Phi(w)}} (-1)^{\ell(w)}(1|w)(\x)\prod_{\alpha \in \Phi(w)} q^{\text{ht}(\alpha)}\x^{2\alpha}  + \sum_{\substack{w \in W \\ e_i \in \Phi(w)}} (-1)^{\ell(w)}(1|w)(\x)\prod_{\alpha \in \Phi(w)} q^{\text{ht}(\alpha)}\x^{2\alpha} \right)
\end{equation}
and the latter sum is 
\begin{align}
& \sum_{\substack{w \in W \\ e_i \not{\in} \Phi(w)}} (-1)^{\ell(w\sigma_i)}(1|w\sigma_i)(\x)\prod_{\alpha \in \Phi(w\sigma_i)} q^{\text{ht}(\alpha)}\x^{2\alpha} \nonumber \\
&=-\sum_{\substack{w \in W \\ e_i \not{\in} \Phi(w)}} (-1)^{\ell(w)}((1|w)|\sigma_i)(\x) q x_i^2 \prod_{\alpha \in \Phi(w)} q^{\text{ht}(\alpha)}(\sigma_i(\x))^{2\alpha}.
\end{align}
At $x_i=\pm q^{-1/2}$, we have $(\sigma_i(\x))^{2\alpha}=\x^{2 \alpha}$, and also $(f|\sigma_i)(\x)=f(\x)$ for any $f$. In this case the two sums cancel, giving the desired result. 

This completes the proof of Proposition (\ref{averaging}). 
\end{proof} 

We will denote the series constructed above as $Z_{\text{avg}}$ in the sequel. A direct computation of a few power series coefficients is sufficient to show that $Z_{\text{avg}}$ cannot satisfy the four axioms in the affine case. However, by Corollary (\ref{onevariable}), a series $Z(\x)$ satisfying the axioms must be $Z_{\text{avg}}$ multiplied by a power series in one variable. We cannot obtain the ratio explicitly, because we do not have enough information about the zeroes of $Z_{\text{avg}}$, but we can still obtain analytic information about $Z$ from $Z_{\text{avg}}$.

\begin{cor}\label{cont1}
If the series $Z(\x)$ satisfies the four axioms, then it has meromorphic continuation to $|\x^{\alpha_0}|<q^{-\text{ht}(\alpha_0)}$.
\end{cor}

\begin{proof}
By Axiom (\ref{dominance}), $Z(\x)$ is absolutely convergent if all $x_i<q^{-1}$ Hence the ratio $Z(\x)/Z_{\text{avg}}(\x)$ is a meromorphic function in the same domain. By Proposition (\ref{onevariable}), this ratio is a power series in the variable $\x^{\alpha_0}$, so it must be meromorphic for $|\x^{\alpha_0}|<q^{-\text{ht}(\alpha_0)}$. Multiplying it by $Z_{\text{avg}}(\x)$ again proves the corollary.
\end{proof}

%Title/short description of chapter
\chapter{Residues of $Z(\x)$}\label{chap:4}
\section{Definition of the Residue}

In this section, we restrict to affine Kac-Moody Weyl groups $W$. We make an additional assumption for technical reasons, that the Dynkin diagram $\Gamma$ of $W$ is two-colorable: we have a partition of the vertices $\lbrace 1, \ldots n+1 \rbrace=S\amalg T$ such that any two adjacent vertices belong to opposite sets. This only rules out $\tilde{A}_n$ for $n$ even. We will abbreviate the restriction of an $n+1$-tuple to one set or the other--for example, if $\x=(x_1, \ldots x_{n+1})$, then $\x_S=(x_i)_{i \in S}$. 

Let $Z(x_1, \ldots x_{n+1})$ be a series satisfying the four axioms. Then $Z(\x)$ has a polar hyperplane along each $x_i=q^{-1}$. We will study the behavior of a particular residue of $Z$, namely
\begin{equation}
R(\x_S):=(-q)^{|T|} \text{Res}_{\x_T=(q^{-1}, \ldots q^{-1})} Z(\x).
\end{equation}
This residue exists and is meromorphic for $|\x_S^{\alpha_0|_S}|<q^{-\text{ht}(\alpha_0|_S)}$, by Corollary (\ref{cont1}). However, we will approach the residue as another formal power series. Taking a residue may not be a well-defined operation on an arbitrary formal power series in several variables. However, we may multiply our series $Z(\x)$ by $1-qx_i$ and then evaluate it at $x_i=q^{-1}$. By Proposition ({\ref{axiomsimplyfes}), this evaluation only requires taking finite sums of series coefficients, so it is well-defined. This is the meaning of $-q\text{Res}_{x_i=q^{-1}}$.

The series $Z(\x)$ can be recovered from $R(\x_S)$, but $R(\x_S)$ has properties which make it more amenable to computation. In particular, the coefficients of $R$ are multiplicative, not twisted multiplicative, so $R$ has an Euler product expression. We will identify a symmetry in this expression which corresponds to Axiom (\ref{localglobal}). The Euler product for $R(\x_S)$ can be separated into a diagonal and an off-diagonal factor. Assuming that the off-diagonal factor is fixed, and using the local-global symmetry, we will prove Theorem (\ref{main}), that $Z(\x)$ exists and is uniquely characterized by the four axioms. We will explicitly compute the off-diagonal factor in the following section. 

First let us verify the statement that $Z(\x)$ can be recovered from $R(\x_S)$, using only the functional equations. Indeed, if $Z(\x)$ and $Z'(\x)$ are two series satisfying the functional equations (\ref{oddglobalfe}) and (\ref{evenglobalfe}), then Corollary (\ref{onevariable}) gives 
\begin{equation}
\frac{Z(\x)}{Z'(\x)}=F(\x^{\alpha_0}),
\end{equation}
a power series in one variable. The ratio of the corresponding residues $R(\x_S)$ and $R'(\x_S)$ must be essentially the same:
\begin{equation}
\frac{R(\x_S)}{R'(\x_S)}=F(q^{-\text{ht}(\alpha_0|_T)} \x_S^{\alpha_0|_S}),
\end{equation}
where $\alpha_0|_S$ denotes the projection of the root $\alpha_0$ onto the space spanned by simple roots in $S$. If we compare diagonal parts of all these series, the result is the same. Let $Z_{\text{diag}}(x)$, $Z'_{\text{diag}}(x)$ denote the diagonal parts of the series $Z$, $Z'$, with $x$ substituted for the variable $\x^{\alpha_0}$, and similarly let $R_{\text{diag}}(x)$, $R'_{\text{diag}}(x)$ be the diagonal parts of $R$, $R'$, with $x$ substituted for $\x_S^{\alpha_0|_S}$. Then 
\begin{equation}
\frac{Z_{\text{diag}}(x)}{Z'_{\text{diag}}(x)}=\frac{R_{\text{diag}}(q^{\text{ht}(\alpha_0|_T)}x)}{R'_{\text{diag}}(q^{\text{ht}(\alpha_0|_T)}x)}
\end{equation}
or, equivalently, 
\begin{equation} \label{residuetodiagonal}
G(x)=\frac{R_{\text{diag}}(q^{\text{ht}(\alpha_0|_T)}x)}{Z_{\text{diag}}(x)}
\end{equation}
is a one-variable power series  depending only on the functional equations, not on the choice of $Z(\x)$. It can be thought of as the $R_{\text{diag}}(q^{\text{ht}(\alpha_0|_T)}x)$ if we choose $c_{0,\ldots 0}(q)=1$ and $c_{m\alpha_0}(q)=0$ for all other diagonal coefficients of $Z(\x)$. Starting with $R_{\text{diag}}(q^{\text{ht}(\alpha_0|_T)}x)$ and multiplying by $G(x)^{-1}$ gives $Z_{\text{diag}}(x)$, which is known by Proposition (\ref{diagonalcoeffs}) to determine $Z(\x)$.

\section{A Symmetry}

Next we prove a proposition relating the coefficients of $R(\x_S)$ to those of $Z(\x)$. Given a list of indices $a_1, \ldots a_{n+1}$, let $A(i)$ denote $\sum\limits_{j \sim i} a_j$. Let $N(i)$ denote $\#\lbrace j\sim i \rbrace$.

\begin{prop} \label{residueformula1}
Suppose that $Z(\x)=\sum\limits_{a_1, \ldots a_{n+1}} c_{a_1, \ldots a_{n+1}}(q) x_1^{a_1}\cdots x_{n+1}^{a_{n+1}}$. Then the coefficient of $\prod_{i\in S} x_i^{a_i}$ in $R(\x_S)$ is $c_{a_1', \ldots a_{n+1}'}(q)q^{-\sum_{i\in S} a_i N(i)} $ where 
\begin{equation} \label{extenda}
a_i'=\bigg{\lbrace}\begin{array}{cc} a_i & i \in S \\ A(i) & i \in T\end{array}.
\end{equation}
In particular, only terms with $A(i)$ even for all $i \in T$ appear in $R(\x_S)$. 
\end{prop}

\begin{proof}
Fix all indices except one $a_i$ for $i \in T$. Proposition (\ref{axiomsimplyfes}) implies that the series in $x_i$, $\sum\limits_{a_i=0}^{\infty} c_{a_1, \ldots a_{n+1}}(q) x_1^{a_1}\cdots x_{n+1}^{a_{n+1}}$, can be written as
\begin{equation}
\left\lbrace\begin{array}{cc} \sum\limits_{a_i=0}^{A(i)-1} c_{a_1, \ldots a_{n+1}}(q) x_1^{a_1}\cdots x_{n+1}^{a_{n+1}} & A(i) \text{ odd} \\ \sum\limits_{a_i=0}^{A(i)-1} c_{a_1, \ldots a_{n+1}}(q) x_1^{a_1}\cdots x_{n+1}^{a_{n+1}}+ \frac{c_{a_1, \ldots A(i), \ldots a_{n+1}}(q) x_1^{a_1}\cdots x_i^{A(i)}\cdots x_{n+1}^{a_{n+1}}}{1-qx} & A(i) \text{ even} \end{array}\right. .
\end{equation}
Then taking $-q \text{Res}_{x_i=q^{-1}}$ gives $0$ if $A(i)$ is odd, and $c_{a_1, \ldots A(i), \ldots a_{n+1}}(q) x_1^{a_1}\cdots q^{-A(i)}\cdots x_{n+1}^{a_{n+1}}$ if $A(i)$ is even. Repeating this process for all $i \in T$ gives 
\begin{align}
R(\x_S)&=\sum_{(a_i)_{i\in S}} c_{a_1', \ldots a_{n+1}'}(q) q^{-\sum_{i\in T} A(i)} \prod_{i \in S} x_i^{a_i} \nonumber \\
&=\sum_{(a_i)_{i \in S}} c_{a_1', \ldots a_{n+1}'}(q) q^{-\sum_{i\in S} a_i N(i)} \prod_{i \in S} x_i^{a_i}
\end{align}
where $a_i'$ are as in the statement of the proposition. The rearrangements of power series implicit in this proof are only reorderings of finite sums, by Proposition (\ref{axiomsimplyfes}). They can also be justified by the absolute convergence of $Z(\x)$ for all $x_i< q^{-1}$. 
\end{proof}

Since we have shown that 
\begin{equation}
(-q)^{|T|} \text{Res}_{\x_T=(q^{-1}, \ldots q^{-1})} Z(\x)=(-q)^{|T|} \text{Res}_{\x_T=(q^{-1}, \ldots q^{-1})} Z_{A(i) \text{ even}}(\x)
\end{equation}
for all $i \in T$, we may apply the even $\sigma_i$ functional equations to the residue. By equation (\ref{evenglobalfe}), we have:
\begin{align} \label{residuefe}
R(\x_S) &= (-q)^{|T|} \text{Res}_{\x_T=(q^{-1}, \ldots q^{-1})} \prod_{i \in T} \frac{1-x_i^{-1}}{1-qx_i} Z_{A(i) \text{ even }\forall i\in T}((\prod_{i \in T} \sigma_i)(\x)) \nonumber \\
&=(1-q)^{|T|}Z_{A(i) \text{ even }\forall i\in T}(x_1', \ldots x_{n+1}') 
\end{align}
where 
\begin{equation}
x_i'=\left\lbrace \begin{array}{cc} q^{-N(i)/2} x_i & \text{if } i\in S \\ 1 & \text{if } i \in T \end{array} \right. .
\end{equation} 
Multiplying the series $Z(\x)$ by $(1-qx_i)$ and evaluating at $x_i=1$ is a well-defined operation--again, by Proposition (\ref{axiomsimplyfes}), it only involves taking finite sums of coefficients. This allows us to prove the following alternate formula for $R(\x_S)$:

\begin{prop} \label{residueformula2}
We have 
\begin{equation}
R(\x_S)=\sum_{\substack{(f_i)_{i\in S} \\ f_i \in \mathbb{F}_q[t] \text{ monic}}} H(f_1', \ldots f_{n+1}') \prod_{i \in S} (q^{-N(i)/2} x_i)^{\deg f_i}
\end{equation}
where
\begin{equation} \label{extendf}
f_i'=\left\lbrace \begin{array}{cc} f_i & i \in S \\ \prod_{j \sim i} f_j & i \in T \end{array} \right. .
\end{equation}
In particular, only $|S|$-tuples of polynomials with $\prod_{j\sim i} f_j$ a perfect square for all $i \in T$ contribute to this sum.
\end{prop}
\begin{proof}
First, we observe that if $(f_1', \ldots f_{n+1}')$ has this form, then 
\begin{equation}
H(f_1', \ldots f_{n+1}')=\prod_{p| f_1' \cdots f_{n+1}'} H(p^{v_p(f_1')}, \ldots p^{v_p(f_{n+1}')})
\end{equation}
as all the quadratic residue symbols in Axiom (\ref{twistedmult}) multiply to $1$. For $i \in T$, we see that $v_p(f_i')=\sum_{j\sim i} v_p(f_j')$. By Proposition (\ref{axiomsimplyfes}), this means that $H$ vanishes unless $\sum_{j\sim i} v_p(f_j')$ is even for all $i \in T$, i.e. unless all the $\prod_{j\sim i} f_j$ are perfect squares.

Fix all polynomials except one $f_i$ for $i \in T$. We assume that $\sum_{j \sim i} \deg f_j=A(i)$ is even, and we write $\prod_{j \sim i}f_j$ as $g_0 g_1^2$ for $g_0$ squarefree of even degree. As in the proof of Proposition (\ref{axiomsimplyfes}), the series $\sum_{f_i} H(f_1, \ldots f_{n+1}) x_i^{\deg f_i}$ matches $L(x_i, \chi_{g_0})$ up to multiplication by a polynomial. In particular, if $g_0 \neq 1$, then the series has a trivial zero at $x_i=1$. If $g_0=1$, i.e. $\prod_{j\sim i} f_j$ is a perfect square, then the series $\sum_{f_i} H(f_1, \ldots f_{n+1}) x_i^{\deg f_i}$ matches the zeta function $(1-qx_i)^{-1}$ up to multiplication by a correction polynomial of the form:
\begin{equation}
\prod_{p|g_1} \left(H(p^{v_p(f_1)}, \ldots p^{v_p(\prod_{j\sim i} f_j)}, \ldots p^{v_p(f_{n+1})}) x_i^{v_p(\prod_{j\sim i} f_j) \deg p}+(1-x_i^{\deg p})(\ldots)\right)
\end{equation}
Multiplying the series by $1-qx_i$ and evaluating at $x_i=1$ gives $H(f_1, \ldots \prod_{j\sim i} f_j, \ldots f_{n+1})$. Repeating this process for all $i \in T$ verifies the proposition. The rearrangements of power series implicit in this computation are again only reorderings of finite sums.
\end{proof}

The first part of this proof indicates that the function $(f_i)_{i \in S} \mapsto H(f_1', \ldots f_{n+1}')$ with $f_i'$ defined by equation (\ref{extendf}) is multiplicative, not just twisted multiplicative: that is, if $\gcd(\prod_{i\in S} f_i, \prod_{i\in S} g_i)=1$, then 
\begin{equation}
H(f_1'g_1', \ldots f_{n+1}'g_{n+1}')=H(f_1', \ldots f_{n+1}')H(g_1', \ldots g_{n+1}').
\end{equation} 
This, together with Proposition (\ref{residueformula2}), gives $R(\x_S)$ an Euler product expression:
\begin{equation}
R(\x_S)=\prod_{\substack{p \in \mathbb{F}_q[t] \\  \text{prime}}} \left(\sum_{\substack{(a_i)_{i\in S} \\ a_i \geq 0}} H(p^{a_1'}, \ldots p^{a_{n+1}'}) \prod_{i \in S} (q^{-N(i)/2} x_i)^{a_i \deg p} \right)
\end{equation}
where the $a_i'$ are as in equation (\ref{extenda}). 

Indeed, if we take the formula for $R(\x)$ of Proposition (\ref{residueformula1}):
\begin{equation}
R(\x_S)=\sum_{\substack{(a_i)_{i\in S} \\ a_i \geq 0}} c_{a_1', \ldots a_{n+1}'}(q) \prod_{i \in S} (q^{-N(i)} x_i)^{a_i}
\end{equation}
and make the substitution $q \mapsto q^{-\deg p}$, $x_i \mapsto (q^{1-N(i)/2} x_i)^{\deg p}$, by Axiom (\ref{localglobal}) we obtain the Euler factor:
\begin{equation}
\sum_{\substack{(a_i)_{i\in S} \\ a_i \geq 0}} H(p^{a_1'}, \ldots p^{a_{n+1}'}) \prod_{i \in S} (q^{-N(i)/2} x_i)^{a_i \deg p}.
\end{equation}
Hence the residue has the following property, which is the analogue of the Local-Global Axiom:
\begin{property} \label{symmetry}
If $(1-q^{\mu}\prod\limits_{i \in S} x_i^{\nu_i})^{-\lambda}$ is a factor of $R(\x_S)$, then $(1-q^{1-\mu+\sum\limits_{i\in S} (1-N(i)/2)\nu_i}\prod\limits_{i \in S} x_i^{\nu_i})^{-\lambda}$ is also a factor.
\end{property}
Note that any power series in $q, x_1, \ldots x_{n+1}$ can be expressed, at least formally, as a product of factors of this form. 

Moreover, given an arbitrary power series $Z(\x)=\sum\limits_{a_1, \ldots a_{n+1}} c_{a_1, \ldots a_{n+1}}(q) x_1^{a_1}\cdots x_{n+1}^{a_{n+1}}$ satisfying the functional equations (\ref{oddglobalfe}) and (\ref{evenglobalfe}), there exists a choice of local weights $H(f_1, \ldots f_n)$ which produces this series. Write the residue of $Z(\x)$ as a product:
\begin{equation} \label{formalprod}
R(\x_S)=\prod_{k=1}^{\infty} (1-q^{\mu(k)}\prod_{i \in S} x_i^{\nu_i(k)})^{-\lambda(k)}
\end{equation}
and use this to define the Euler factor
\begin{equation}
\sum_{\substack{(a_i)_{i\in S} \\ a_i \geq 0}} H(p^{a_1'}, \ldots p^{a_{n+1}'}) \prod_{i \in S} (q^{-N(i)/2} x_i)^{a_i \deg p}=\prod_{k=1}^{\infty} (1-q^{(\mu(k)-1)\deg p}\prod_{i \in S} x_i^{\nu_i(k)\deg p})^{-\lambda}.
\end{equation}
This series determines $Z_p(\x)=\sum\limits_{a_1, \ldots a_{n+1}} H(p^{a_1}, \ldots p^{a_{n+1}}) x_1^{a_1\deg p}\cdots x_{n+1}^{a_{n+1}\deg p}$ by the local functional equations just as $R(\x_S)$ determines $Z(\x)$. Then $H(f_1, \ldots f_{n+1})$ can be evaluated for all $f_i \in \mathbb{F}_q[t]$ by Axiom (\ref{twistedmult}). By construction, we have 
\begin{equation}
\sum_{f_1, \ldots f_{n+1}} H(f_1, \ldots f_{n+1}) x_1^{\deg f_1} \cdots x_{n+1}^{\deg f_{n+1}} = \sum_{a_1, \ldots a_{n+1}} c_{a_1, \ldots a_{n+1}}(q) x_1^{a_1} \cdots x_{n+1}^{a_{n+1}}.
\end{equation}
Axiom (\ref{twistedmult}) will be satisfied, but Axioms (\ref{localglobal}) and (\ref{dominance}) may not. In fact, Axiom (\ref{localglobal}) is equivalent to Property (\ref{symmetry}). This is because Proposition (\ref{residueformula1}), which uses only the functional equations, can be applied to $R(\x)$, and if the residue coefficients satisfy the local-global property, then all coefficients do. 

\section{Proof of the Main Theorem}

Since $R(\x_S)$ is determined by functional equations up to a power series in the variable $\x_S^{\alpha_0|_S}$, it is natural to factor it into diagonal and off-diagonal terms. Write $R(\x_S)$ as a product of factors $(1-q^{\mu}\prod\limits_{i \in S} x_i^{\nu_i})^{-\lambda}$, and let $R(\x_S)= R_0(\x_S) R_1(\x_S)$, where $R_1(\x_S)$ collects the factors where $(\nu_i)$ is a multiple of $\alpha_0|_S$, and $R_0(\x_S)$ collects the others. The off-diagonal factor $R_0(\x_S)$ is the same for every series satisfying the functional equations, but the diagonal factor $R_1(\x_S)$ may vary. In the next section, we will give an explicit formula for $R_0(\x_S)$, which satisfies Property (\ref{symmetry}).

For now, let us assume that $R_0(\x_S)$ is fixed; let $R_{0, \text{ diag}}(x)$ denote its diagonal part, with $x$ substituted for $\x_S^{\alpha_0|_S}$. As in equation (\ref{residuetodiagonal}), we have:
\begin{equation} \label{residuetodiagonal}
G(x)R_{0,\text{ diag}}(q^{\text{ht}(\alpha_0|_T)}x)^{-1}Z_{\text{diag}}(x)= R_1(q^{\text{ht}(\alpha_0|_T)}x)
\end{equation}
We will use this equation to show that there is a unique choice of $Z(\x)$ satisfying the four axioms. The series $G(x)R_{0,\text{ diag}}(q^{\text{ht}(\alpha_0|_T)}x)^{-1}$ is fixed. In order for the axioms to be satisfied, we must have:
\begin{condition}
The coefficients of $Z_{\text{diag}}(x)$ satisfy Axiom (\ref{dominance}): the coefficient of $x^a$ is divisible by $q^{1+\text{ht}(\alpha_0)/2}$. By Corollary (\ref{diagonaldominance}), this implies the Dominance Axiom for all coefficients of $Z(\x)$.
\end{condition}
\begin{condition}
$R_1(q^{\text{ht}(\alpha_0|_T)}x)$ satisfies Property (\ref{symmetry}): if $(1-q^{\mu}x^{\nu})^{-\lambda}$ is a factor, then $(1-q^{1-\mu+\nu\text{ht}(\alpha_0)}x^{\nu})^{-\lambda}$ is also a factor. This, together with the symmetry for $R_0(\x_S)$, implies Axiom (\ref{localglobal}).
\end{condition} 

Recall that $G(x)$ is essentially the diagonal part of the residue of a series with diagonal coefficients $c_{0, \ldots 0}(q)=1$, $c_{m\alpha_0}(q)=0$ for all nonzero $m$. These coefficients trivially satisfy Axiom (\ref{dominance}), so $G(x)$ satisfies a version of Dominance: its $x^a$ coefficient is a polynomial in $q$, supported in degrees between $a\text{ht}(\alpha_0|_S)/2$ and $a\text{ht}(\alpha_0)$. In fact, it is possible to prove a stronger lower bound, but for our purposes here the upper bound is sufficient. By the formulas of the next section, the $x^a$ coefficient of $R_{0,\text{ diag}}(q^{\text{ht}(\alpha_0|_T)}x)^{-1}$ is also a polynomial in $q$ of degree at most $a\text{ht}(\alpha_0)$. Thus we may write 
\begin{equation}
G(x)R_{0,\text{ diag}}(q^{\text{ht}(\alpha_0|_T)}x)^{-1}=\prod_k(1-q^{\mu(k)}x^{\nu(k)})^{-\lambda(k)}
\end{equation}
with $0<\mu(k)<\nu(k)\text{ht}(\alpha_0)$ in each factor. If we separate the product as 
\begin{equation}
\prod_{\mu(k) \leq \nu(k)\text{ht}(\alpha_0)/2} (1-q^{\mu(k)}x^{\nu(k)})^{-\lambda(k)} \prod_{\mu(k) > \nu(k)\text{ht}(\alpha_0)/2} (1-q^{\mu(k)}x^{\nu(k)})^{-\lambda(k)} 
\end{equation}
Then the only way to satisfy the two conditions above is to set
\begin{equation} 
R_1(q^{\text{ht}(\alpha_0|_T)}x)=\prod_{\mu(k) \leq \nu(k)\text{ht}(\alpha_0)/2} (1-q^{\mu(k)}x^{\nu(k)})^{-\lambda(k)}(1-q^{1-\mu(k)+\nu(k)\text{ht}(\alpha_0)}x^{\nu(k)})^{-\lambda(k)}
\end{equation}
and
\begin{equation}
Z_{\text{diag}}(x)=\prod_{\mu(k) \leq \nu(k)\text{ht}(\alpha_0)/2} (1-q^{1-\mu(k)+\nu(k)\text{ht}(\alpha_0)}x^{\nu(k)})^{-\lambda(k)} \prod_{\mu(k) > \nu(k)\text{ht}(\alpha_0)/2} (1-q^{\mu(k)}x^{\nu(k)})^{\lambda(k)}.
\end{equation}
This proves the existence and uniqueness of a series $Z(\x)$ with $R(\x_S)=R_0(\x_S)R_1(\x_S)$, satisfying the four axioms. All that remains is to make the computation of $R_0(\x_S)$. 

%Title/short description of chapter
\chapter{Residue Formulas}\label{chap:5}
\section{The Off-Diagonal Part $R_0$ of the Residue}

In this section, we continue to study the residue $R(\x_S)$ of a multiple Dirichlet series for an affine Weyl group $W$ whose Dynkin diagram is two-colorable. We prove an explicit formula for the off-diagonal factor $R_0(\x_S)$ as a product of function field zeta functions. The proof is analogous to that of Theorem (\ref{diagonalcoeffs}), relying upon the group of functional equations satisfied by the residue, which must be studied separately for each affine type. We then give conjectural formulas for the diagonal factor  $R_1(\x_S)$ in each type. These conjectures are supported by computational evidence. In the following section, we will prove the conjecture for $\tilde{A}_n$ when $n$ is odd.

Before stating the formula for $R_0(\x_S)$, let us describe all its possible poles. Recall that the possible poles of the Chinta-Gunnells averaged series are given by the Weyl denominator 
\begin{equation}
D(\x)=\prod_{\alpha \in \Phi^+} (1-q^{\text{ht}(\alpha)+1}\x^{2 \alpha})^{-1}.
\end{equation}
Of course, some of these factors may be canceled in the numerator. $Z(\x)$ differs from this series by a power series in one variable $\x^{\alpha_0}$. Hence the possible poles of $Z(\x)$ are given by the same Weyl denominator plus poles of the form $\x^{\alpha_0}=c$.

It follows that the possible poles of $R(\x_S)$ either are of the form $\x_S^{\alpha_0|_S}=c$, or correspond to factors $(1-q^{\text{ht}(\alpha|_S)-\text{ht}(\alpha|_T)+1}\x_S^{2 \alpha|_S})^{-1}$ for $\alpha\in \Phi^+$. In the latter case we consider the orbit of $\alpha$ under the group $<\sigma_i: i \in T> = (\mathbb{Z}/2 \mathbb{Z})^{|T|}$, which leaves $\alpha|_S$ unchanged. Let us denote this orbit as $[\alpha]$ and its size as $2^{t(\alpha)}$. 

By Proposition (\ref{dist}), for $\alpha\in \Phi$, $|\sigma_i(\alpha)-\alpha|_{\infty}=0$, $1$, or $2$, and is only $2$ in the case when $\alpha = m\alpha_0 \pm e_i$. In the second case, $\alpha|_S=m\alpha_0|_S$, so these poles correspond to diagonal factors and can be ignored.

In the first case, $t(\alpha)$ is the size of the set $\lbrace i \in T: |\sigma_i(\alpha)-\alpha|_{\infty}=1\rbrace$, which is precisely the set of $i\in T$ where $\sum_{j\sim i} (\alpha)_j$ is odd. In the orbit $[\alpha]$, $\text{ht}(\alpha|_T)$ varies around the mean value $\sum_{i \in T} \sum_{j \sim i} (\alpha)_j/2 = \sum_{j \in S} N(j) (\alpha)_j /2$ by increments of $1/2$. More precisely, we have a product formula describing all possible off-diagonal poles of $R(\x_S)$:
\begin{equation}
\prod_{\substack{[\alpha] \alpha\in \Phi^+ \\ \alpha \neq m\alpha_0 \pm e_i \text{ for }i\in T}} \prod_{u=0}^{t(\alpha)} (1-q^{\sum_{j \in S} (1-N(j)/2)(\alpha)a_j -t(\alpha)/2+u+1} \x_S^{2\alpha|_S})^{-\binom{t(\alpha)}{u}}.
\end{equation}
Notice that this product does not satisfy Property (\ref{symmetry}), though it does have a similar symmetry. We will see that many of the possible poles of the residue are canceled by zeroes in the numerator. The correct formula is as follows:

\begin{prop} \label{R0}
We have
\begin{align} \label{R0formula}
R_0(\x_S) &= \prod_{[\alpha], \, \alpha\in \Phi^+, \, t(\alpha)=0} (1-q^{(\sum_{j \in S} (1-N(j)/2)(\alpha)_j+1)/2} \x_S^{\alpha|_S})^{-1} \nonumber \\
&\prod_{\substack{[\alpha], \, \alpha\in \Phi^+, \, t(\alpha)>0 \\ \alpha \neq m\alpha_0 \pm e_i \text{ for }i\in T}} \prod_{u=0}^{t(\alpha)-1} (1-q^{\sum_{j \in S} (1-N(j)/2)(\alpha)_j -t(\alpha)/2+u+1} \x_S^{2\alpha|_S})^{-\binom{t(\alpha)-1}{u}}.
\end{align}
The residue of any power series $Z(\x)$ satisfying functional equations (\ref{oddglobalfe}) and (\ref{evenglobalfe}) is $R_0(\x_S)$ multiplied by a series in one variable $\x_S^{\alpha_0|_S}$. 
\end{prop}

Before proving this proposition, a few comments: first, $R_0(\x_S)$ has half of its possible poles in any given orbit $[\alpha]$, and it does satisfy the symmetry (\ref{symmetry}). It can be directly verified that the $\x_S^{a\alpha_0|_S}$ coefficient of $R_0(q^{\text{ht}(\alpha_0|_T)}\x_S)$ is a polynomial in $q$ of degree at most $a \text{ht}(\alpha_0)$ as asserted in the previous section. We may extend $R_0(\x_S)$ to a meromorphic function for $|\x_S^{\alpha_0|_S}|<q^{(\text{ht}(\alpha_0|_T)-\text{ht}(\alpha_0|_S))/2}$, which is compatible with the maximal domain of meromorphic continuation for $Z(\x)$. The proof or meromorphicity is the same as the proof for $D(\x)$ in Section (\ref{chap:3}). Finally, note that the formula (\ref{R0formula}) may include some diagonal factors; this is purely a matter of notational convenience, and the statement remains true if these factors are removed, but we will assume they are included.

\begin{proof}
The strategy of the proof is to give a group of functional equations which the residue must satisfy, and which determine it up to a power series in one variable. We can then check that the residue formula (\ref{R0formula}) satisfies these functional equations. The proof also implicitly uses Proposition (\ref{averaging}) to assert that some residue of a series satisfying the functional equations exists, which must then match $R_0(\x_S)$ multiplied by a diagonal series. There is not, for example, an extra functional equation with no solutions. 

The functional equations of the residue all correspond to elements $w$ of the normalizer of $<\sigma_i: i \in T>$ in $W$. Therefore, they permute the orbits $[\alpha]$ of this group, and map orbits of the form $[m\alpha_0 +e_i]$ for $i\in T$ to orbits of the same form. All but finitely many orbits of positive roots map to orbits of positive roots. From this we will see that $R_0(\x_S)$ has a $w$ functional equation with a scalar cocycle which is a product over orbits $[\alpha]$ for $\alpha\in\Phi(w)$. We must directly compute the cocycle given by the functional equation for $w$ acting on an arbitrary residue $R(\x_S)$, and verify that it matches the cocycle on $R_0(\x_S)$.

The group of functional equations is type-dependent, so the proof has several cases, involving long, repetitive calculations. We will only give complete details in the first few cases. The most complicated formulas were verified in Mathematica.

The simplest examples of functional equations satisfied by the residue come from vertices $i\in S$ with $N(i)=2$. We will label the adjacent vertices are $i-1$ and $i+1$. Then the functional equation of $R(\x_S)$ is derived from $\sigma_i\sigma_{i-1}\sigma_{i+1}\sigma_i$.
  
In the notation of Proposition (\ref{averaging}), we have 
\begin{align}
Z(\x)&=(Z|\sigma_i\sigma_{i-1}\sigma_{i+1}\sigma_i)(\x)  \nonumber \\
&=\frac{1}{16}\sum_{\delta_1, \delta_2, \delta_3, \delta_4 \in \lbrace 0, 1 \rbrace} (-1)^{\delta_2+\delta_3} q^{-2} x_i^{-4} x_{i-1}^{-2} x_{i+1}^{-2} \left(\frac{x_i-1}{1-qx_i}+(-1)^{\delta_4}q^{-1/2}\right) \nonumber \\
&\left(\frac{(-1)^{\delta_4}q^{1/2}x_{i-1}x_i-1}{1-(-1)^{\delta_4}q^{3/2}x_{i-1}x_i}+(-1)^{\delta_3}q^{-1/2}\right) \left(\frac{(-1)^{\delta_4}q^{1/2}x_i x_{i+1}-1}{1-(-1)^{\delta_4}q^{3/2} x_i x_{i+1}}+(-1)^{\delta_2}q^{-1/2}\right) \nonumber \\ 
& \left(\frac{(-1)^{\delta_2+\delta_3}q x_{i-1} x_i x_{i+1}-1}{1-(-1)^{\delta_2+\delta_3}q^2 x_{i-1}x_i x_{i+1}}+(-1)^{\delta_1}q^{-1/2}\right) Z(\epsilon_i^{\delta_1} \sigma_i \epsilon_{i-1}^{\delta_2} \sigma_{i-1} \epsilon_{i+1}^{\delta_3} \sigma_{i+1} \epsilon_i^{\delta_4} \sigma_i(\x))
\end{align}
where
\begin{equation}
(\epsilon_i^{\delta_1} \sigma_i \epsilon_{i-1}^{\delta_2} \sigma_{i-1} \epsilon_{i+1}^{\delta_3} \sigma_{i+1} \epsilon_i^{\delta_4} \sigma_i(\x))_j= \left\lbrace \begin{array}{cc} (-1)^{\delta_2+\delta_3}q^{-2} x_{i-1}^{-1} x_i^{-1} x_{i+1}^{-1} & j=i \\ (-1)^{\delta_1+\delta_2+\delta_3+\delta_4} x_{i \mp 1} & j=i \pm 1 \\ (-1)^{\delta_2} q x_i x_{i-1} x_j & j\sim i-1, j\neq i \\ (-1)^{\delta_3} q x_i x_{i+1} x_j & j\sim i+1, j\neq i \\ x_j & \text{otherwise} \end{array} \right. .
\end{equation} 
If we now take the residue, only terms with $(-1)^{\delta_1+\delta_2+\delta_3+\delta_4}=1$ will contribute at all, and by Proposition (\ref{residueformula1}), all powers of $-1$ in $Z$ will cancel out. We define the resulting transformation as $\tau_i(\x_S)$, given by
\begin{equation}
(\tau_i(\x_S))_j=\left\lbrace \begin{array}{cc} x_j^{-1} & j=i \\ x_i x_j & j\sim\sim i \\ x_j & \text{otherwise} \end{array} \right.
\end{equation}
where $\sim\sim$ denotes vertices of distance two apart in the Dynkin diagram. The result is that $R(\x_S)$ has a functional equation with scalar cocycle:
\begin{equation} \label{2neighborfe}
R(\x_S)=(*)R(\tau_i(\x_S))
\end{equation}
where 
\begin{align}
(*)&=\frac{1}{16}\sum_{\delta_2, \delta_3, \delta_4 \in \lbrace 0, 1 \rbrace} (-1)^{\delta_2+\delta_3} q^{-6} x_i^{-4}
\left(\frac{x_i-1}{1-qx_i}+(-1)^{\delta_4}q^{-1/2}\right) \nonumber \\
& \left(\frac{(-1)^{\delta_4}q^{-1/2}x_i-1}{1-(-1)^{\delta_4}q^{1/2}x_i}+(-1)^{\delta_3}q^{-1/2}\right)
\left(\frac{(-1)^{\delta_4}q^{-1/2}x_i-1}{1-(-1)^{\delta_4}q^{1/2} x_i}+(-1)^{\delta_2}q^{-1/2}\right) \nonumber \\
&\left(\frac{(-1)^{\delta_2+\delta_3}q^{-1}x_i-1}{1-(-1)^{\delta_2+\delta_3}x_i}+(-1)^{\delta_2+\delta_3+\delta_4}q^{-1/2}\right) \nonumber \\
&=\frac{(1-x_i^{-2})(1-qx_i^{-2})}{(1-x_i^2)(1-qx_i^2)}.
\end{align}

We now check that $R_0(\x_S)$ as defined in formula (\ref{R0formula}) satisfies this functional equation. Because $\sigma_i\sigma_{i-1}\sigma_{i+1}\sigma_i$ lies in the normalizer of $<\sigma_i: i \in T>$, it permutes the orbits $[\alpha]$, mapping each $[\alpha]$ to $\tau_i[\alpha]:=[\sigma_i\sigma_{i-1}\sigma_{i+1}\sigma_i \alpha]$, with $t(\alpha)=t(\sigma_i\sigma_{i-1}\sigma_{i+1}\sigma_i \alpha)$. Orbits $[m\alpha_0 +e_i]$ for $i\in T$ map to orbits of the same form. Furthermore, for each $[\alpha]$ with $t(\alpha)=0$, we have 
\begin{equation}
(1-q^{(\sum_{j \in S} (1-N(j)/2)(\alpha)_j+1)/2} \tau_i(\x_S)^{\alpha|_S})^{-1}=(1-q^{(\sum_{j \in S} (1-N(j)/2)(\tau_i(\alpha))_j+1)/2} \x_S^{\tau_i(\alpha)|_S})^{-1}
\end{equation}
and for each $[\alpha]$ with $t(\alpha)>0$, we have
\begin{align}
&\prod_{u=0}^{t(\alpha)-1} (1-q^{\sum_{j \in S} (1-N(j)/2)(\alpha)_j -t(\alpha)/2+u+1} \tau_i(\x_S)^{2\alpha|_S})^{-\binom{t(\alpha)-1}{u}} \nonumber \\ &=\prod_{u=0}^{t(\tau_i(\alpha))-1} (1-q^{\sum_{j \in S} (1-N(j)/2)(\tau_i(\alpha))_j -t(\tau_i(\alpha))/2+u+1} \x_S^{2\tau_i(\alpha)|_S})^{-\binom{t(\alpha)-1}{u}}
\end{align}
so $\tau_i$ permutes the factors of $R_0(\x_S)$. The only exception is the orbit $[e_i]$, with $t(e_i)=2$, which maps to $[-e_i]$. No other positive orbit appearing in $R_0(\x_S)$ becomes negative under $\tau_i$. To account for the factors gained and lost from this orbit, we have 
\begin{equation}
R_0(\x_S)=\frac{(1-x_i^{-2})(1-qx_i^{-2})}{(1-x_i^2)(1-qx_i^2)} R_0(\tau_i(\x_S))
\end{equation}
which indeed matches the functional equation (\ref{2neighborfe}) above. 

With this functional equation in hand, we can prove the proposition in type $\tilde{A}_n$, $n$ odd. Let us label the vertices of the Dynkin diagram from $1$ to $n+1$ modulo $n+1$, as shown,

\includegraphics[scale=.5]{A_n}

\noindent and without loss of generality let $S$ be the set of odd-numbered vertices, and $T$ the set of even-numbered vertices. For each odd $i$, we have a functional equation $\tau_i$ as above. These generate a group of symmetries isomorphic to the Weyl group of $\tilde{A}_{(n-1)/2}$.

Suppose $R(\x_S)$ and $R'(\x_S)$ are two residues satisfying these functional equations. Then because the cocycle is scalar, the ratio of the two residues is invariant under all the transformations $\tau_i$. If we write 
\begin{equation}
\frac{R(\x_S)}{R'(\x_S)}=\sum_{a_1, a_3, \ldots a_n \geq 0} d_{a_1, a_3, \ldots a_n} x_1^{a_1}x_3^{a_3}\cdots x_n^{a_n}
\end{equation}
then the $\tau_i$ functional equation yields a coefficient relation 
\begin{equation}
d_{\ldots a_{i-2}, a_i, a_{i+2}, \ldots}=d_{\ldots a_{i-2}, a_{i-2}+a_{i+2}-a_i, a_{i+2}, \ldots}
\end{equation}
Any non-diagonal coefficient $d_{a_1, a_3, \ldots a_n}$ can be reduced repeatedly by such relations, so it must be $0$. Hence the ratio $\frac{R(\x_S)}{R'(\x_S)}$ is a power series in the variable $x_1 x_3\cdots x_n$.

For $n=3$, the the argument must be modified slightly. Here $\tau_i$ is the transformation:
\begin{equation}
(\tau_i(\x_S))_j=\left\lbrace \begin{array}{cc} x_j^{-1} & j=i \\ x_i^2 x_j & j \neq i \end{array} \right.
\end{equation}
which induces the coefficient relations $d_{a_1, a_3}=d_{2a_3-a_1, a_3}=d_{a_1, 2a_1-a_3}$. The rest of the proof is similar.

In the remaining types, the residue $R(\x_S)$ will have a somewhat more complicated functional equation for each vertex $i\in S$ with $N(i)=3$. If we label the adjacent vertices $i-1$, $i+1$, and $i+2$ then this functional equation is induced from the Weyl group element given by $\sigma_i \sigma_{i-1} \sigma_{i+1} \sigma_{i+2} \sigma_{i} \sigma_{i-1} \sigma_{i+1} \sigma_{i+2} \sigma_{i}$. After a computation similar to the one above, we find a transformation $\tau_i(\x_S)$ given by 
\begin{equation}
(\tau_i(\x_S))_j=\left\lbrace \begin{array}{cc} x_j^{-1} & j=i \\ x_i^2 x_j & j\sim\sim i \\ x_j & \text{otherwise} \end{array} \right.
\end{equation}
and a functional equation as follows:
\begin{equation} \label{3neighborfe}
R(\x_S)=\frac{(1-x_i^{-2})(1-qx_i^{-2})^3(1-q^2x_i^{-2})}{(1-q^{-1}x_i^2)(1-x_i^2)^3(1-qx_i^2)} R(\tau_i(\x_S)).
\end{equation}
The cocycle is originally a sum of $256$ terms, but there are many cancellations.

To verify that $R_0(\x_S)$ satisfies this functional equation, we use the fact that the transformation $\sigma_i \sigma_{i-1} \sigma_{i+1} \sigma_{i+2} \sigma_{i} \sigma_{i-1} \sigma_{i+1} \sigma_{i+2} \sigma_{i}$ normalizes $<\sigma_i: i \in T>$, and hence permutes all but finitely many factors of $R_0(\x_S)$. The only orbits $[\alpha]$ of positive roots which map to orbits of negative roots are $[e_i]$ with $t(e_i)=3$ and $[2e_i+e_{i-1}+e_{i+1}+e_{i+2}]$ with $t(2e_i+e_{i-1}+e_{i+1}+e_{i+2})=0$. We see that $[e_i]$ contributes $\frac{(1-x_i^{-2})(1-qx_i^{-2})^2(1-q^2x_i^{-2})}{(1-q^{-1}x_i^2)(1-x_i^2)^2(1-qx_i^2)}$ to the cocycle, and $[2e_i+e_{i-1}+e_{i+1}+e_{i+2}]$ contributes $\frac{1-qx_i^{-2}}{1-x_i^2}$, so equation (\ref{3neighborfe}) is satisfied.

We now prove the proposition for the case of $\tilde{D}_n$. We label the vertices of the Dynkin diagram as follows:

\includegraphics[scale=.5]{D_n}

\noindent Suppose $n$ is even and $S=\lbrace 3, 5, \ldots n-1 \rbrace$. Let $d_{a_3, a_5, \ldots a_{n-1}}$ denote a coefficient of the ratio of two residues. The $\tau_i$ invariance yields relations $d_{\ldots a_{i-2}, a_i, a_{i+2}, \ldots}=d_{\ldots a_{i-2}, a_{i-2}+a_{i+2}-a_i, a_{i+2}, \ldots}$, $d_{a_3, a_5, \ldots}=d_{2a_5-a_3, a_5, \ldots}$, and $d_{\ldots a_{n-3}, a_{n-1}}=d_{\ldots a_{n-3}, 2a_{n-3}-a_{n-1}}$. If $a_{n-1}>a_{n-3}$, $a_3>a_5$, or any other $a_i>(a_{i-2}+a_{i+2})/2$, these relations allow the coefficient $d_{a_3, a_5, \ldots a_{n-1}}$ to be reduced repeatedly, until it is $0$. Hence the ratio must be a diagonal series in the variable $x_3 x_5 \ldots x_{n-1}$. We have already shown that $R_0(\x_S)$ satisfies the correct $\tau_i$ functional equations, so the proposition is verified.

In the remaining cases of type $\tilde{D}$, we require an additional functional equation, which will be denoted $\tau_{1, 2}$ or $\tau_{n, n+1}$. The transformation $\tau_{1,2}$ derives from $\sigma_1 \sigma_2 \sigma_3 \sigma_1 \sigma_2$, and $\tau_{n,n+1}$ derives from $\sigma_n \sigma_{n+1} \sigma_{n-1} \sigma_n \sigma_{n+1}$. We will describe the functional equation for $\tau_{1,2}$ only, because $\tau_{n,n+1}$ is similar. Assume that $1, 2 \in S$. As in the cases above, we compute 
\begin{equation}
(\tau_{1, 2}(\x_S))_j=\left\lbrace \begin{array}{cc} x_2^{-1} & j=1 \\ x_1^{-1} & j=2 \\ x_1 x_2 x_4 & j=4 \\ x_j & \text{otherwise} \end{array} \right.
\end{equation}
and the functional equation 
\begin{equation}
R(\x_S)=\frac{(1-x_1^{-2})(1-x_2^{-2})(1-x_1^{-1}x_2^{-1})}{(1-qx_1^2)(1-qx_2^2)^3(1-qx_1x_2)} R(\tau_{1,2}(\x_S)).
\end{equation}
We see that $R_0(\x_S)$ satisfies this functional equation because $\sigma_1 \sigma_2 \sigma_3 \sigma_1 \sigma_2$ normalizes $<\sigma_i: i \in T>$, and permutes the orbits $[\alpha]$ of positive roots, except for $[e_1]$ and $[e_2]$, and $[e_1+e_2+e_3]$, which map to $[-e_2]$, $[-e_1]$, and $[-e_1-e_2-e_3]$ respectively.

For $\tilde{D}_n$ with $n$ even and $S=\lbrace 1, 2, 4, 6, \ldots n-2, n, n+1 \rbrace$, the functional equations $\tau_{1, 2}$, $\tau_4$, $\tau_6$,... $\tau_{n-2}$, $\tau_{n, n+1}$ are not quite sufficient to determine the residue up to a diagonal series. If 
$d_{a_1, a_2, a_4, \ldots a_{n-2}, a_n, a_{n+1}}$ is a nonzero coefficient of the ratio of two residues, then the functional equations give relations which imply that $a_1+a_2=a_4=a_6=\cdots=a_n-2=a_n+a_n+1$, but not that $a_1=a_2=a_n=a_n+1$. We require supplemental functional equations, valid in this case only. 

For $i\in \lbrace 1, 2 \rbrace$ and $j \in \lbrace n, n+1 \rbrace$, the extra functional equation corresponds to 
\begin{align}
(\sigma_i \sigma_3 \sigma_4 \cdots \sigma_{n/2})(\sigma_j \sigma_{n-1} \sigma_{n-2} \cdots \sigma_{2+n/2}) \nonumber \\
(\sigma_{1+ n/2}) \nonumber \\
(\sigma_{2+n/2} \sigma_{3+n/2} \cdots \sigma_{n-1}\sigma_j) (\sigma_{n/2} \sigma_{n/2-1} \cdots \sigma_3 \sigma_i).  
\end{align}
It induces a transformation $\tau_{i,j}(\x_S)$ given by 
\begin{equation}
(\tau_{i,j}(\x_S))_k=\left\lbrace \begin{array}{cc} x_4^{-1}x_6^{-1}\cdots x_{n-2}^{-1}x_j^{-1} & k=i \\ x_i^{-1}x_4^{-1}x_6^{-1}\cdots x_{n-2}^{-1} & k=j \\ x_1 x_2 x_4 x_6 \cdots x_{n-2} x_j & k \in \lbrace 1, 2 \rbrace, k \neq i \\ x_4 x_6 \cdots x_{n-2} x_n x_{n+1} & k \in \lbrace n, n+1 \rbrace, k \neq j \\ x_k & \text{otherwise} \end{array} \right. .
\end{equation}
The functional equation is 
\begin{equation}
R(\x_S)=(*)R(\tau_i(\x_S))
\end{equation}
with the scalar cocycle
\begin{align}
(*)=&\left(\frac{1-x_i^{-2}}{1-qx_i^2}\right)\left(\frac{1-x_j^{-2}}{1-qx_j^2}\right)\left( \frac{1-(x_i x_4 x_6 \cdots x_{n-2} x_j)^{-1}}{1-qx_i x_4 x_6 \cdots x_{n-2} x_j}\right) \nonumber \\ &\prod\limits_{\substack{4\leq k \leq n-2 \\ k \text{ even}}} \left(\frac{1-(x_i x_4 x_6 \cdots x_k)^{-2}}{1-q(x_i x_4 x_6 \cdots x_k)^2}\right)\left( \frac{1-(x_k x_{k+2} \cdots x_{n-2} x_j)^{-2}}{1-q(x_k x_{k+2} \cdots x_{n-2} x_j)^2}\right).
\end{align}
$R_0(\x_S)$ satisfies this functional equation because, as always, the underlying transformation normalizes $<\sigma_i: i \in T>$, and permutes the orbits $[\alpha]$. The orbits of positive roots which map to negative roots are: $[e_i], [e_i+e_3+e_4], \ldots [e_i+e_3+\cdots+e_{n-2}], [e_j], [e_j+e_{n-1}+e_{n-2}], \ldots [e_j+e_{n-1}+\cdots+e_4], [e_i+e_3+\cdots +e_{n-1}+e_j]$, and their images are $[-e_j-e_{n-1}-\cdots-e_4], [-e_j-e_{n-1}-\cdots-e_6], \ldots [-e_j], [-e_i-e_3-\cdots-e_{n-2}], [-e_i-e_3-\cdots-e_{n-4}],\ldots [-e_i], [-e_i-e_3-\cdots -e_{n-1}-e_j]$ respectively.

The extra functional equation $\tau_{1, n}$ gives rise to a coefficient relation which allows us to reduce $d_{a_1, a_2, a_4, \ldots a_{n-2}, a_n, a_{n+1}}$ if $a_1+a_n>a_2+a_{n+1}$, and the other $\tau_{i, j}$ have similar results. The full group of functional equations determines the residue up to a series in the variable $x_1 x_2 x_4^2 x_6^2\cdots x_{n-2}^2 x_n x_{n+1}$.

For $\tilde{D}_n$ with $n$ odd, we may take $S=\lbrace 1, 2, 4, 6, \ldots n-1 \rbrace$ without loss of generality. The functional equations $\tau_{1, 2}$, $\tau_4$, $\tau_6$, ... $\tau_{n-2}$ imply that $a_1+a_2=a_4=a_6=\cdots=a_{n-1}$ for any nonzero coefficient $d_{a_1, a_2, a_4, a_6, \ldots a_{n-1}}$ in the ratio of two residues. Again, we require extra functional equations, valid only in this case, to show that $a_1=a_2$.

For $i\in\lbrace 1, 2 \rbrace$, a functional equation comes from $\sigma_i\sigma_3\cdots\sigma_{n-1}\sigma_n\sigma_{n+1}\sigma_{n-1}\cdots \sigma_3\sigma_i$. We set 
\begin{equation}
(\tau_i(\x_S))_j=\left\lbrace \begin{array}{cc} x_i^{-1}x_4^{-2}x_6^{-2}\cdots x_{n-2}^{-2} & j=i \\ x_i^2 x_j x_4^2 x_6^2 \cdots x_{n-2}^2 & j \in \lbrace 1, 2 \rbrace, j \neq i \\ x_j & \text{otherwise} \end{array} \right.
\end{equation}
and compute the functional equation:
\begin{equation}
R(\x_S)=(*)R(\tau_i(\x_S))
\end{equation}
where
\begin{align}
(*)=&\left(\frac{1-x_i^{-2}}{1-qx_i^2}\right) \left(\frac{1-qx_i^{-2} x_4^{-2} x_6^{-2}\cdots x_{n-1}^{-2}}{1-x_i^2x_4^2x_6^2\cdots_{n-1}^2}\right) \nonumber \\ &\prod_{\substack{4\leq k\leq n-1 \\ k \text{ even}}}\left(\frac{1-x_i^{-2}x_4^{-2}x_6^{-2}\cdots x_k^{-2}}{1-qx_i^2x_4^2x_6^2\cdots x_k^2}\right)\left(\frac{1-qx_i^{-2} x_4^{-2} x_6^{-2} \cdots x_{k-2}^{-2} x_k^{-4} x_{k+2}^{-4}\cdots x_{n-1}^{-4}}{1-x_i^2 x_4^2 x_6^2 \cdots x_{k-2}^2 x_k^{4} x_{k+2}^{4}\cdots x_{n-1}^{4}}\right). 
\end{align}
The orbits of positive roots which map to negative roots under this transformation are $[e_i], [e_i+e_3+e_4], \ldots [e_i+e_3+\cdots+e_{n-1}], [e_i+e_4+\cdots+e_{n-2}+2e_{n-1}+e_n+e_{n+1}], \ldots [e_i+2e_3+\cdots+2e_{n-1}+e_n+e_{n+1}]$, all of which have $t=1$ except for $[e_i+e_3+\cdots+e_{n-1}]$ with $t=2$. Hence $R_0(\x_S)$ satisfies this functional equation.

If, for example, $a_1>a_2$, then the $\tau_1$ relation reduces $d_{a_1, a_2, a_4, a_6, \ldots a_{n-1}}$. It follows that $R(\x_S)$ is determined up to a diagonal series.

For type $\tilde{E}_6$, we label the vertices of the Dynkin diagram:

\includegraphics[scale=.5]{E_6}

\noindent Let $S=\lbrace 2, 4, 6 \rbrace$. Then the $\tau_2, \tau_4, \tau_6$ functional equations of (\ref{2neighborfe}) generate a group isomorphic to the Weyl group $\tilde{A}_2$. If $d_{a_2, a_4, a_6}$ is a coefficient in the ratio of two residues, we have 
\begin{equation}
d_{a_2, a_4, a_6}=d_{a_4+a_6-a_2, a_4, a_6}=d_{a_2, a_2+a_6-a_4, a_6}=d_{a_2, a_4, a_2+a_4-a_6}
\end{equation}
and so any coefficient without $a_2=a_4=a_6$ can be reduced to $0$. Hence the functional equations determine the residue up to a diagonal series.

On the other hand, if $S= \lbrace 1, 3, 5, 7 \rbrace$, then we need supplemental functional equations. We will define transformations $\tau_1, \tau_3, \tau_5$ corresponding to $\sigma_1\sigma_2\sigma_7\sigma_4\sigma_6\sigma_7\sigma_2\sigma_1$,  $\sigma_3\sigma_4\sigma_7\sigma_2\sigma_6\sigma_7\sigma_4\sigma_3$,  $\sigma_5\sigma_6\sigma_7\sigma_2\sigma_4\sigma_7\sigma_6\sigma_5$ respectively, given by
\begin{equation}
(\tau_i(\x_S))_j=\left\lbrace \begin{array}{cc} x_i^{-1}x_7^{-2} & j=i \\ x_j x_i x_7 & j \neq i, 7 \\ x_j & j=7 \end{array} \right. .
\end{equation}
Then the functional equations are
\begin{equation}
R(\x_S)=\frac{(1-x_i^{-2})(1-x_i^{-2}x_7^{-2})(1-qx_i^{-2}x_7^{-2})(1-qx_i^{-2}x_7^{-4})}{(1-qx_i^2)(1-qx_i^2x_7^2)(1-x_i^2x_7^2)(1-x_i^2x_7^4)}R(\tau_i(\x_S)).
\end{equation}
The positive orbits which map to negative orbits are $[e_i]$, with $t=1$ $[e_i+e_{i+1}+e_7]$, with $t=2$, and $[e_i+e_2+e_4+e_6+2e_7]$, with $t=1$. From this, we see that $R_0(\x_S)$ satisfies the functional equation.

These extra functional equations, together with $\tau_7$ as in equation (\ref{3neighborfe}), determine the residue up to a series in $x_1x_3x_5 x_7^3$. The relation given by the $\tau_1$ functional equation, for example, is $d_{a_1, a_3, a_5, a_7}=d_{a_3+a_5-a_1, a_3, a_5, a_7+a_3+a_5-a_1}$, which allows any coefficient with $a_1>(a_3+a_5)/2$ to be reduced.

For type $\tilde{E}_7$, we label the Dynkin diagram:

\includegraphics[scale=.5]{E_7}

\noindent If $S=\lbrace 2, 6, 8 \rbrace$, then the functional equations $\tau_2, \tau_6$ as in (\ref{2neighborfe}) and $\tau_8$ as in (\ref{3neighborfe}) determine the residue up to a diagonal series in $x_2x_6x_8^2$. 

If $S=\lbrace 1, 3, 4, 5, 7 \rbrace$, then we need four extra functional equations. First, for $i \in \lbrace 1, 5 \rbrace$, let $\tau_{i, 4}$ correspond to $\sigma_i \sigma_{i+1} \sigma_{i+2} \sigma_4 \sigma_{8} \sigma_4 \sigma_{i+2} \sigma_{i+1} \sigma_i$. We set
\begin{equation}
(\tau_{i,4}(\x_S))_j=\left\lbrace \begin{array}{cc} x_{i+2}^{-1}x_4^{-1} & j=i \\ x_{i+2}^{-1}x_i^{-1} & j=4 \\ x_j x_i x_{i+2} x_4 & j \in \lbrace 1, 5 \rbrace, j \neq i \\ x_j & j=3, 7 \end{array} \right.
\end{equation}
and obtain the functional equation:
\begin{equation}
R(\x_S)=\frac{(1-x_i^{-2})(1-x_4^{-2})(1-x_i^{-2}x_{i+2}^{-2})(1-x_{i+2}^{-2}x_4^{-2})(1-x_i^{-1}x_{i+2}^{-1}x_4^{-1})}{(1-qx_i^2)(1-qx_4^2)(1-qx_i^2x_{i+2}^2)(1-qx_{i+2}^2x_4^2)(1-qx_i x_{i+2}x_4)}R(\tau_i(\x_S)).
\end{equation}
$R_0(\x_S)$ satisfies this functional equation. The positive orbits which map to negative orbits are $[e_i], [e_4], [e_i+e_{i+1}+e_{i+2}], [e_{i+2}+e_8+e_4], [e_i+e_{i+1}+e_{i+2}+e_8+e_4]$.

Next, we have $\tau_{1, 5}$ corresponding to $\sigma_1\sigma_2\sigma_3\sigma_5\sigma_6\sigma_7\sigma_8\sigma_7\sigma_6\sigma_5\sigma_3\sigma_2\sigma_1$, given by
\begin{equation}
(\tau_{1,5}(\x_S))_j=\left\lbrace \begin{array}{cc} x_{3}^{-1}x_5^{-1}x_7^{-1} & j=1 \\ x_{1}^{-1}x_3^{-1}x_7^{-1} & j=5 \\ x_1 x_3 x_4 x_5 x_7 & j=4 \\ x_j & j=3, 7 \end{array} \right.
\end{equation}
yielding the functional equation:
\begin{equation}
R(\x_S)=(*)R(\tau_i(\x_S)).
\end{equation}
with
\begin{align}
(*)=& \left(\frac{1-x_1^{-2}}{1-qx_1^2}\right)\left( \frac{1-x_5^{-2}}{1-qx_5^2}\right)\left( \frac{1-x_1^{-2}x_3^{-2}}{1-qx_1^2x_3^2}\right)\left( \frac{1-x_5^{-2}x_7^{-2}}{1-qx_5^2x_7^2}\right) \nonumber\\ &\left(\frac{1-x_1^{-2}x_3^{-2}x_7^{-2}}{1-qx_1^2x_3^2x_7^2}\right)\left( \frac{1-x_3^{-2}x_5^{-2}x_7^{-2}}{1-qx_3^2x_5^2x_7^2}\right)\left( \frac{1-x_1^{-1}x_3^{-1}x_5^{-1}x_7^{-1}}{1-qx_1x_3x_5x_7}\right) .
\end{align}
$R_0(\x_S)$ satisfies this functional equation. The positive orbits which map to negative orbits are $[e_1], [e_1+e_2+e_3], [e_5], [e_5+e_6+e_7], [e_1+e_2+e_3+e_8+e_7], [e_5+e_6+e_7+e_8+e_3],[e_1+e_2+e_3+e_5+e_6+e_7+e_8]$.

Finally, $\tau_4$ corresponds to $\sigma_4\sigma_8\sigma_3\sigma_2\sigma_7\sigma_6\sigma_8\sigma_3\sigma_4\sigma_7\sigma_8\sigma_3\sigma_4\sigma_7\sigma_8\sigma_2 \sigma_3\sigma_6\sigma_7\sigma_8\sigma_4$, and is given by
\begin{equation}
(\tau_4(\x_S))_j=\left\lbrace \begin{array}{cc} x_{3}^{-2}x_4^{-1}x_7^{-2} & j=4 \\ x_{j}x_3^2x_4^2x_7^2 & j=1, 5 \\ x_j & j=3, 7 \end{array} \right. .
\end{equation}
The functional equation is 
\begin{equation}
R(\x_S)=(*)R(\tau_i(\x_S)).
\end{equation}
with
\begin{align}
(*)=& \left(\frac{1-x_4^{-2}}{1-qx_4^2}\right)\left( \frac{1-x_3^{-2}x_4^{-2}}{1-qx_3^2x_4^2}\right)\left( \frac{1-x_4^{-2}x_7^{-2}}{1-qx_4^2x_7^2}\right) \nonumber\\ & \left( \frac{1-qx_3^{-2}x_4^{-2}x_7^{-2}}{1-x_3^2x_4^2x_7^2}\right) \left(\frac{1-x_3^{-2}x_4^{-2}x_7^{-2}}{1-qx_3^2x_4^2x_7^2}\right)^3 \left( \frac{1-q^{-1}x_3^{-2}x_4^{-2}x_7^{-2}}{1-q^2x_3^2x_4^2x_7^2}\right) \nonumber \\&\left( \frac{1-x_3^{-4}x_4^{-2}x_7^{-2}}{1-qx_3^4x_4^2x_7^2}\right) \left( \frac{1-x_3^{-2}x_4^{-2}x_7^{-4}}{1-qx_3^2x_4^2x_7^4}\right) \left( \frac{1-x_3^{-4}x_4^{-2}x_7^{-4}}{1-qx_3^4x_4^2x_7^4}\right) 
\end{align}
The positive orbits which map to negative orbits under $\tau_4$ are $[e_4], [e_4+e_7+e_8], [e_3+e_4+e_8], [e_3+e_4+e_7+e_8], [e_2+2e_3+e_4+e_7+2e_8], [e_3+e_4+e_6+2e_7+2e_8], [e_2+2e_3+e_4+e_6+2e_7+2e_8], [e_2+2e_3+2e_4+e_6+2e_7+3e_8]$, so $R_0(\x_S)$ satisfies this functional equation.

If we have a nonzero coefficient $d_{a_1, a_3, a_4, a_5, a_7}$ of the ratio of two residues, the $\tau_3$ and $\tau_5$ relations imply $2a_3 \leq a_1+a_4+a_7$ and $2a_7 \leq a_3+a_4+a_5$. The $\tau_{1, 4}$ and $\tau_{7, 4}$ relations imply $a_1+a_4 \leq a_7$ and $a_4+a_5 \leq a_3$. The $\tau_{1, 5}$ relation implies $a_1+a_5 \leq a_4$. The $\tau_4$ relation implies $a_4 \leq a_1+a_5$. Together, these prove that $(a_1, a_3, a_4, a_5, a_7)$ is proportional to $(1, 3, 2, 1, 3)$. 

The final case is $\tilde{E}_8$, where we label the vertices of the Dynkin diagram:

\includegraphics[scale=.5]{E_8}

\noindent If $S= \lbrace 1, 5, 7, 9 \rbrace$, then we have functional equations $\tau_5, \tau_7$ as in (\ref{2neighborfe}) and $\tau_9$ as in (\ref{3neighborfe}). We also have a functional equation $\tau_1$ derived from $\sigma_1\sigma_2\sigma_9\sigma_3\sigma_8\sigma_9\sigma_2\sigma_1$, which is analogous to $\tau_1$ in the case of $\tilde{E}_6$. The computation of the functional equation and the proof that $R_0(\x)$ satisfies it are identical to the argument in the $\tilde{E}_6$ proof above. If $d_{a_1, a_5, a_7, a_9}$ is a nonzero coefficient in the ratio of two residues, $\tau_5$ implies that $2a_5 \leq a_7$, $\tau_7$ implies that $2a_7 \leq a_5+a_9$, $\tau_9$ implies that $a_9 \leq a_1+a_7$, and $\tau_1$ implies that $2a_1 \leq a_7$. Hence, $(a_1, a_5, a_7, a_9)$ must be proportional to $(1, 1, 2, 3)$. 

If $S=\lbrace 2, 3, 4, 6, 8 \rbrace$, then we have functional equations $\tau_2, \tau_6, \tau_8$ as in (\ref{2neighborfe}). We also will use two supplemental functional equations. The first, $\tau_3$, corresponds to the Weyl group element $\sigma_3\sigma_9\sigma_2\sigma_1\sigma_8\sigma_7\sigma_9\sigma_2\sigma_3\sigma_8\sigma_9\sigma_2\sigma_3\sigma_8\sigma_9\sigma_7 \sigma_8 \sigma_1\sigma_2\sigma_9\sigma_3$, and is analogous to $\tau_4$ in the $\tilde{E}_7$ case. The second, $\tau_{3, 4}$ corresponds to $\sigma_3\sigma_4\sigma_5\sigma_6\sigma_7\sigma_8\sigma_9\sigma_8\sigma_7\sigma_6\sigma_5\sigma_4\sigma_3$. It defines a transformation 
\begin{equation}
(\tau_{3,4}(\x_S))_j=\left\lbrace \begin{array}{cc} x_{4}^{-1}x_6^{-1}x_8^{-1} & j=3 \\ x_{3}^{-1}x_6^{-1}x_8^{-1} & j=4 \\ x_2 x_3 x_4 x_6 x_8 & j=2 \\ x_j & j=6, 8 \end{array} \right.
\end{equation}
and a functional equation  
\begin{equation}
R(\x_S)=(*)R(\tau_i(\x_S)).
\end{equation}
with
\begin{align}
(*)=& \left(\frac{1-x_3^{-2}}{1-qx_3^2}\right)\left( \frac{1-x_4^{-2}}{1-qx_4^2}\right)\left( \frac{1-x_4^{-2}x_6^{-2}}{1-qx_4^2x_6^2}\right) \left( \frac{1-x_3^{-2}x_8^{-2}}{1-qx_3^2x_8^2}\right) \nonumber\\ & \left(\frac{1-x_4^{-2}x_6^{-2}x_8^{-2}}{1-qx_4^2x_6^2x_8^2}\right) \left( \frac{1-x_3^{-2}x_6^{-2}x_8^{-2}}{1-qx_3^2x_6^2x_8^2}\right) \left( \frac{1-x_3^{-1}x_4^{-1}x_6^{-1}x_8^{-1}}{1-qx_3x_4x_6x_8}\right) 
\end{align}
$R_0(\x_S$ satisfies this functional equation. The positive orbits which map to negative orbits are $[e_3], [e_4], [e_3+e_8+e_9], [e_4+e_5+e_6], [e_3+e_6+e_7+e_8+e_9], [e_4+e_5+e_6+e_7+e_8], [e_3+e_4+e_5+e_6+e_7+e_8+e_9]$.

If $d_{a_2, a_3, a_4, a_6, a_8}$ is a nonzero coefficient in the ratio of residues, then we have $2a_2 \leq a_3+a_8$, $a_3 \leq a_6$, $a_3+a_4 \leq a_2$, $2a_6 \leq a_4+a_8$, and $2a_8 \leq a_2+a_3+a_6$. We leave it as an exercise to show that $(a_2, a_3, a_4, a_6, a_8)$ is proportional to $(4, 3, 1, 3, 5)$. 

This completes the computation in all simply-laced affine types.
\end{proof}

With this proposition, our main theorem (\ref{main}) is also completely proven.

\section{The Diagonal Part $R_1$ of the Residue}

We now state conjectures for the diagonal part of the residue, $R_1(\x_S)$. This is a power series in one variable $\x_S^{\alpha_0|_S}$, which, like $R_0(\x_S)$, can be written as an infinite product of function field zeta functions. The exact form of $R_1(\x_S)$ depends on the type and on the set $S$. 

\begin{conj} \label{R1}
We have 
\begin{align} \label{R1Formula}
R_1(\x_S)=\prod_{m=0}^{\infty} & (1-q^{(m+1)(\text{ht}(\alpha_0|_S)-\text{ht}(\alpha_0|_T))} \x^{(2m+2) \alpha_0|_S})^{-|T|} \nonumber \\
&(1-q^{(m+1) (\text{ht}(\alpha_0|_S)-\text{ht}(\alpha_0|_T))+1}  \x^{(2 m+2) \alpha_0|_S})^{-|T|} \nonumber \\
&(1-q^{(m+1/2) (\text{ht}(\alpha_0|_S)-\text{ht}(\alpha_0|_T))} \x^{(2 m+1) \alpha_0|_S})^{-\lambda} \nonumber \\
&(1-q^{(m+1/2) (\text{ht}(\alpha_0|_S)-\text{ht}(\alpha_0|_T))+1}  \x^{(2 m+1) \alpha_0|_S})^{-\lambda}.
\end{align}
The values of $\lambda$ are given in the following table:

\begin{tabular}{l | c | r}
Type & S & $\lambda$ \\
\hline
$\tilde{A}_n$, $n$ odd & - & $1$ \\
$\tilde{D}_n$, $n$ odd & - & $1$ \\
$\tilde{D}_n$, $n$ even & $\lbrace 1, 2, 4, 6, \ldots n-2, n, n+1 \rbrace$ & $0$ \\
$\tilde{D}_n$, $n$ even & $\lbrace 3, 5, 7, \ldots n-1 \rbrace$ & $3$ \\
$\tilde{E}_6$ & $\lbrace 1, 3, 5, 7 \rbrace$ & $0$ \\
$\tilde{E}_6$ & $\lbrace 2, 4, 6 \rbrace$ & $1$ \\
$\tilde{E}_7$ & $\lbrace 1, 3, 4, 5, 7 \rbrace$ & $0$ \\
$\tilde{E}_7$ & $\lbrace 2, 6, 8 \rbrace$ & $2$ \\
$\tilde{E}_8$ & $\lbrace 1, 5, 7, 9 \rbrace$ & $1$ \\
$\tilde{E}_8$ & $\lbrace 2, 3, 4, 6, 8 \rbrace$ & $0$ \\
\end{tabular}
\end{conj}

This conjecture is based upon computational evidence in all types. In the following section, we will prove it for type $\tilde{A}_n$, but the other cases remain open. To conclude this section, we make several remarks.

The formula (\ref{R1Formula}) is an eta-product--that is, a product of four Dedekind eta functions. The classification of all roots in an affine root system as $m\alpha_0+\alpha$ for $\alpha$ in a finite set implies that (\ref{R0formula}) is an eta-product as well. The full residue $R(\x_S)=R_0(\x_S)R_1(\x_S)$ is strongly reminiscent of formulas appearing in the MacDonald identities for affine Weyl groups \cite{M}. One might hope that these identities furnish a more straightforward proof of Proposition (\ref{R0}), or any proof of Conjecture (\ref{R1}). 

If Conjecture (\ref{R1}) is verified, then the residue is meromorphic in the domain $|\x_S^{\alpha_0|_S}|<q^{(\text{ht}(\alpha_0|_T)-\text{ht}(\alpha_0|_S))/2}$. The proof of meromorphicity for $R_0(\x)$ is the same as the proof for $D(\x)$ given in Chapter (\ref{chap:3}), and the proof for $R_1(\x)$ is even more straightforward. This implies the meromorphic continuation of the full multiple Dirichlet series $Z(\x)$ to its largest possible domain $|\x^{\alpha_0}|<q^{-\text{ht}(\alpha_0)/2}$. We see this as follows: let $Z_{\text{avg}}(\x)$ denote the multiple Dirichlet series constructed by averaging over the group of functional equations in Proposition (\ref{averaging}), and let $R_{\text{avg}}(\x_S)$ denote its residue. $Z_{\text{avg}}(\x)$ is known to be meromorphic in the largest possible domain. Recall from the previous section that the ratio of $Z$ to $Z_{\text{avg}}$ is essentially the same as the ratio of $R$ to $R_{\text{avg}}$: if
\begin{equation}
\frac{R(\x_S)}{R_{\text{avg}}(\x_S)}=F(q^{-\text{ht}(\alpha_0|_T)} \x_S^{\alpha_0|_S}),
\end{equation}
then we also have 
\begin{equation}
\frac{Z(\x)}{Z_{\text{avg}}(\x)}=F(\x^{\alpha_0}).
\end{equation}
Since $R$ and $R_{\text{avg}}$ are meromorphic, $F(q^{-\text{ht}(\alpha_0|_T)} \x_S^{\alpha_0|_S})$ must be meromorphic for  $\x_S^{\alpha_0|_S}<q^{(\text{ht}(\alpha_0|_T)-\text{ht}(\alpha_0|_S))/2}$. Therefore, $Z(\x)=Z_{\text{avg}}(\x)F(\x^{\alpha_0})$ is meromorphic for $|\x^{\alpha_0}|<q^{-\text{ht}(\alpha_0)/2}$.

Diaconu and Bucur in the $\tilde{D}_4$ case construct a series whose residue is assumed to be solely the $R_0$ we have computed above, without the $R_1$. Under this assumption, they obtain meromorphic continuation to the optimal domain.

One brief comment on the Eisenstein conjecture: Eisenstein series on Kac-Moody Lie groups are expected to have poles corresponding to all roots, real and imaginary. This phenomenon should be visible in the Whittaker coefficient $Z(\x)$. The poles corresponding to real roots are those which can be deduced from the functional equations alone--they are the poles of $D(\x)$. The averaged series $Z_{\text{avg}}$ has only these poles. The series $Z(\x)$ satisfying the axioms, on the other hand, must have poles corresponding to imaginary roots. This follows from Conjecture (\ref{R1}). The first factor of equation (\ref{R1Formula}) comes from real roots, but the other factors do not. We cannot describe the full family of poles of $Z(\x)$ coming from imaginary roots, because some of them may be canceled in the residue, but we can assert that such poles do exist. This is a piece of evidence that the series $Z(\x)$ is the correct one for the Eisenstein conjecture.

%Title/short description of chapter
\chapter{Computing the Full Residue in Type $\tilde{A}$}\label{chap:6}
\section{Restatement of Results in Type $\tilde{A}$}

Let $W$ be the Weyl group of a simply laced affine root system $\tilde{A}_n$, with $n$ odd. Label the vertices of the Dynkin diagram $1$ to $n+1$ modulo $n+1$:

\includegraphics[scale=.5]{A_n}

\noindent Let $Z(x_1, \ldots x_{n+1})$ be the $\tilde{A}_n$ multiple Dirichlet series satisfying the four axioms. In this chapter we prove Conjecture (\ref{R1}) for this series. That is, we prove the following residue formula.
\begin{align}\label{residue}
&R(x_1, x_3, \ldots x_n):=(-q)^{(n+1)/2} \text{Res}_{x_2=x_4=\cdots=x_{n+1}=1/q} Z(x_1, \ldots x_{n+1}) \nonumber \\ 
&=\prod_{m=0}^{\infty} (1-(x_1 x_3\cdots x_n)^{2m+1})^{-1}(1-q (x_1 x_3\cdots x_n)^{2m+1})^{-1} \nonumber\\ 
&\left(\prod_{i,j \text{ odd, }i \not\equiv j+2} (1-(x_1 x_3\cdots x_n)^{2m}(x_i x_{i+2}\cdots x_j)^2)^{-1}(1-q (x_1 x_3\cdots x_n)^{2m}(x_i x_{i+2}\cdots x_j)^2)^{-1}\right) \nonumber \\ 
&(1-(x_1 x_3\cdots x_n)^{2m+2})^{-(n+1)/2}(1-q (x_1 x_3\cdots x_n)^{2m+2})^{-(n+1)/2}
\end{align}
A priori, we know that this residue is meromorphic for $|x_1 x_3 \cdots x_n|<q^{-(n+1)/2}$, but by this formula, it is actually meromorphic for $|x_1 x_3 \cdots x_n|<1$. It follows that $Z(\x)$ is meromorphic for $|x_1\cdots x_{n+1}|<q^{-(n+1)/2}$. 

Let us recall the results of the previous two chapters, specialized to type $\tilde{A}$. We gave a formula (\ref{residueformula1}) for $R$ in terms of the coefficients $c_{a_1, \ldots a_{n+1}}$ of the original power series $Z$, namely:
\begin{equation}\label{rseries}
R(x_1, x_3, \ldots x_n) = \sum_{a_1, a_3, \ldots a_n} \frac{c_{a_1, a_1+a_3, a_3, a_3+a_5, \ldots, a_n, a_n+a_1}}{q^{2(a_1+a_3+\ldots+a_n)}} x_1^{a_1}x_3^{a_3}\cdots x_n^{a_n}.
\end{equation}
In particular, the nonvanishing coefficients must have $a_1, a_3, \ldots a_n$ all odd or all even, since if $a_i+a_{i+2}$ is odd, then $c_{\ldots a_i, a_i+a_{i+2}, a_{i+2}, \ldots}=0$.

In Proposition (\ref{R0}) of the previous chapter, we gave a formula for the off-diagonal factors of the residue, which in this case match the factors in the third line of (\ref{residue}).
\begin{align}
R_0(x_1, \ldots x_n)=\prod_{m=0}^{\infty}\prod_{\substack{i,j \text{ odd, }\\i \not\equiv j+2}} (1-(x_1 \cdots x_n)^{2m}(x_i \cdots x_j)^2)^{-1}(1-q (x_1\cdots x_n)^{2m}(x_i\cdots x_j)^2)^{-1}
\end{align}
Any residue of a series satisfying the functional equations matches this one up to a diagonal power series in $x_1x_3\cdots x_n$; we must show that the correct diagonal series is given by the second and fourth lines of (\ref{residue}).

Recall equation (\ref{residuetodiagonal}) of chapter \ref{chap:4}, which was used in the proof of the main theorem. This equation is equivalent to the following:
\begin{equation}
G(q^{-(n+1)/2}x)Z_{\text{diag}}(q^{-(n+1)/2}x)= R_{\text{diag}}(x)
\end{equation}
Here $Z_{\text{diag}}$ and $R_{\text{diag}}$ denote the diagonal parts of series and its residue, with $x$ substituted for $x_1x_2\cdots x_{n+1}$ and $x_1 x_3 \cdots x_{n}$ respectively. $G(q^{-(n+1)/2}x)$ is what the diagonal part of the residue would be in a series satisfying the functional equations with diagonal coefficients $c_{0, \ldots 0}(q)=1$ and $c_{m, \ldots m}(q)=0$ for all $m>0$. Its coefficients are $q^{-a(n+1)}c_{a, 2a, a, 2a, \ldots a, 2a}(q)$, again under this assumption about the diagonal coefficients. To avoid confusion with the correct coefficients, we will write $q^{-a(n+1)}p_a(q)$ for a coefficient of $G(q^{-(n+1)/2}x)$.

Now we will employ the other axioms. The dominance axiom states that $c_{a, a, \ldots a}$ is a polynomial in $q$, divisible by $q^{a(n+1)/2+1}$ except when $a=0$; hence $Z_{\text{diag}}(q^{-(n+1)/2}x)=1+O(q)$.  Define 
\begin{align*}
R^{\flat}(x_1, x_3, \ldots x_n):=&\prod_{m=0}^{\infty} (1-(x_1 x_3\cdots x_n)^{2m+1})^{-1} \\
&(\prod_{i,j \text{ odd, }i \not\equiv j+2} (1-(x_1 x_3\cdots x_n)^{2m}(x_i x_{i+2}\cdots x_j)^2)^{-1} \\ 
&(1-(x_1 x_3\cdots x_n)^{2m+2})^{-(n+1)/2}
\end{align*}
and let $R^{\flat}_{\text{ diag}}(x)$ denote the diagonal part of the series $R^{\flat}$. By Property (\ref{symmetry}), the local-global axiom is equivalent to the statement that the factors of $R(x_1, x_3, \ldots x_n)$ come in pairs: $(1-q^b x_1^{a_1} x_3^{a_3}\cdots x_n^{a_n})^{-1}$ with $(1-q^{1-b} x_1^{a_1} x_3^{a_3}\cdots x_n^{a_n})^{-1}$. Therefore, if we can show that $R(x_1, x_3, \ldots x_n)=R^{\flat}(x_1, x_3, \ldots x_n)(1+O(q))$, we will verify all of equation (\ref{residue}).  Moreover, since we have already determined the off-diagonal factors of $R$, it suffices to show that $R_{\text{diag}}(x)=R^{\flat}_{\text{diag}}(x)(1+O(q))$. By equation (\ref{residuetodiagonal}), this means 
\begin{equation}
\sum_a \frac{p_a(q)}{q^{a(n+1)}} x^a = R^{\flat}_{\text{diag}}(x)(1+O(q)).
\end{equation}

\section{A Combinatorial Proof}

More concretely, we must prove the following:
\begin{prop}
The lowest term in $p_a(q)$ has degree $a(n+1)$ and coefficient equal to the coefficient of $x^a$ in $R^{\flat}_{\text{diag}}(x)$.
\end{prop}
\begin{proof}
The proof requires closely examining the combinatorics of the recurrences on coefficients of $Z$. Recall the statement of the recurrence associated to functional equation $\sigma_i$: if $a_{i-1}+a_{i+1}$ is odd, then 
\begin{equation} 
c_{\ldots a_i, \ldots}=q^{a_i-(a_{i-1}+a_{i+1}-1)/2} c_{\ldots a_{i-1}+a_{i+1}-1-a_i, \ldots}
\end{equation}
and if $a_{i-1}+a_{i+1}$ is even, then 
\begin{equation} 
c_{\ldots a_i, \ldots}=q c_{\ldots a_i-1, \ldots} + q^{a_i-(a_{i-1}+a_{i+1})/2}(c_{\ldots a_{i-1}+a_{i+1}-a_i, \ldots}-q c_{a_{i-1}+a_{i+1}-a_i-1})
\end{equation}
or, by applying this recurrence repeatedly, 
\begin{equation}
c_{\ldots a_i, \ldots}=q^{a_i-(a_{i-1}+a_{i+1})/2}(c_{\ldots (a_{i-1}+a_{i+1})/2, \ldots} + \sum_{a_i'=a_{i-1}+a_{i+1}-a_i}^{(a_{i-1}+a_{i+1})/2-1} (c_{\ldots a_i', \ldots} -q c_{\ldots a_i'-1, \ldots})).
\end{equation}

Starting with $c_{a_1, \ldots a_n}$ we will apply the recurrences in the following order: first, reduce as far as possible with the even $\sigma_i$, then reduce the result as far as possible with the odd $\sigma_i$, then reduce that result as far as possible with the even $\sigma_i$, and so on. Any coefficient will eventually be reduced to a linear combination of diagonal coefficients in this way. The lowest term in $p_a(q)$ represents the number of paths from $c_{a, 2a, a, 2a, \ldots, a, 2a}$ to $c_{0, 0,\ldots 0}$ via these recurrences, gaining as small a power of $q$ as possible. 

Given any $c_{a_1, \ldots a_{n+1}}$, assuming without loss of generality that $\sum\limits_{i \text{ even}} a_i \geq \sum\limits_{i \text{ odd}} a_i$, we apply the recurrences $\sigma_i$ for $i$ even to reduce as far as possible. Any coefficient $c_{a_1, a_2', \ldots a_n, a_{n+1}'}$ in the resulting expression now has $\sum\limits_{i \text{ even}} a_i' \leq \sum\limits_{i \text{ odd}} a_i$, and is multiplied by a factor of at least $q^{\sum\limits_{i \text{ even}} a_i - \sum\limits_{i \text{ odd}} a_i}$. If we continue reducing this way until we reach $c_{0, 0, \ldots 0}$, it will be multiplied by a factor of at least $q^{\text{Max}(\sum\limits_{i \text{ even}} a_i, \sum\limits_{i \text{ odd}} a_i)}$. In particular, $p_a(q)=O(q^{a(n+1)})$. This is the correct order since one possible path is $c_{a, 2a, a, 2a, \ldots a, 2a} \to q^{a(n+1)/2} c_{a, 0, a, 0, \ldots a, 0} \to q^{a(n+1)} c_{0, 0, \ldots 0}$.

Because we are only considering the lowest term in $p_a(q)$, we can discard all terms in the $\sigma_i$ recurrence with a factor greater than $q^{a_i-(a_{i-1}+a_{i+1})/2}$. This leads to greatly simplified recurrences: if $a_{i-1}+a_{i+1}$ is even, then 
\begin{equation}
c_{\ldots a_i, \ldots}=q^{a_i-(a_{i-1}+a_{i+1})/2}\sum_{a_i'=a_{i-1}+a_{i+1}-a_i}^{(a_{i-1}+a_{i+1})/2} c_{\ldots a_i', \ldots}
\end{equation}
and if $a_{i-1}+a_{i+1}$ is odd, then
\begin{equation} \label{oddrec}
c_{\ldots a_i, \ldots}=0.
\end{equation}
Given $c_{a_1, \ldots a_{n+1}}$ with $\sum\limits_{i \text{ even}} a_i \geq \sum\limits_{i \text{ odd}} a_i$, suppose for some even $i$ we have $a_i<(a_{i-1}+a_{i+1})/2$. Then this index cannot be reduced via the recurrences, and reducing the other even indices will add a power of $q$ greater than $q^{\sum\limits_{i \text{ even}} a_i - \sum\limits_{i \text{ odd}} a_i}$. Thus such terms $c_{a_1, \ldots a_{n+1}}$ can be discarded. Moreover, suppose for some $i$ we have $a_i<a_{i-1}$. We must have $a_{i+1}<a_i$ or discard this term. We reduce all even indices via the simplified recurrences. Then any coefficient $c_{a_1, a_2' \ldots a_n, a_{n+1}'}$ in the resulting expression will have $a_{i+2}'<a_{i+1}<a_i'$. Any coefficient $c_{a_1'', a_2' \ldots a_n'', a_{n+1}'}$ in the next step will have $a_{i+3}''<a_{i+2}'<a_{i+1}''$, and so on. In particular, we cannot find a path to $c_{0, \ldots 0}$ this way. Hence $c_{a_1, \ldots a_{n+1}}$ can be discarded. This leads to a further simplification of the recurrence for $a_{i-1}+a_{i+1}$ even:
\begin{equation} \label{evenrec}
c_{\ldots a_i, \ldots}=q^{a_i-(a_{i-1}+a_{i+1})/2}\sum_{a_i'=a_{i-1}+a_{i+1}-a_i}^{\text{Min}(a_{i-1},a_{i+1})} c_{\ldots a_i', \ldots}
\end{equation}
It would be interesting to know whether the simplified recurrences (\ref{oddrec}) and (\ref{evenrec}) somehow correspond to simplified functional equations. 

We have now reduced the problem of computing the first coefficient in $p_a(q)$ to counting chains of indices:
\begin{align*}
&a_1, &&a_2, &&a_3, &&a_4, &&\ldots &&a_n, &&a_{n+1} \\
&a_1', &&a_2', &&a_3', &&a_4', &&\ldots &&a_n', &&a_{n+1}' \\
&a_1'', &&a_2'', &&a_3'', &&a_4'', &&\ldots &&a_n'', &&a_{n+1}'' \\
&\cdots &&\cdots &&\cdots &&\cdots &&\cdots &&\cdots &&\cdots \\
&a_1^{(\ell)}, &&a_2^{(\ell)}, &&a_3^{(\ell)}, &&a_4^{(\ell)}, &&\ldots &&a_n^{(\ell)}, &&a_{n+1}^{(\ell)} \\
\end{align*}
such that:
\begin{condition} \label{ends} 
we have the boundary conditions $(a_1, a_2, \ldots a_n, a_{n+1})=(a, 2a, \ldots a, 2a)$ and $(a_1^{(\ell)}, a_2^{(\ell)}, \ldots a_n^{(\ell)}, a_{n+1}^{(\ell)})=(0, 0 \ldots 0, 0)$
\end{condition}
\begin{condition}
$a_i^{(j)}=a_i^{(j+1)}$ if $i$ is odd and $j$ is even, or if $i$ is even and $j$ is odd.
\end{condition}
\begin{condition} \label{same} 
$a_i^{(j)}+a_{i+2}^{(j)}$ is even for all $i, j$. 
\end{condition}
\begin{condition} \label{ineq} 
For $i$ odd and $j$ even, or $i$ even and $j$ odd, $a_{i-1}^{(j-1)}+a_{i+1}^{(j-1)}-a_{i}^{(j-2)} \leq a_i^{(j)} \leq \text{Min}(a_{i-1}^{(j-1)},a_{i+1}^{(j-1)})$.
\end{condition}
Note that the indices $i$ are still numbered modulo $n+1$ here.

We will rephrase this counting problem once before solving it. For $i$ modulo $n+1$ odd, and $1 \leq j \leq \ell$, let $d_i^{(j)}=a_{i+j+1}^{(j-1)}-a_{i+j}^{(j)}$. Condition (\ref{ineq}) implies that all $d_i^{(j)}$ are nonnegative and $d_i^{(j+1)} \leq d_i^{(j)}$. Thus a chain of indices as above gives rise to an $(n+1)/2$-tuple of integer partitions $d_i^{(1)}\geq d_i^{(2)} \geq \cdots \geq d_i^{(\ell)} \geq 0$. By condition (\ref{same}), we have $d_i^{(j)}+d_{i+2}^{(j)}$ even for all $i, j$. Condition (\ref{ends}) is equivalent to $\sum\limits_{j=1}^\ell d_{i-2j}^{(j)}=a$ for all $i$.

We can reconstruct the chain of indices $a_i^{(j)}$ from the partitions $d_i^{(j)}$: for $i$ odd and $j$ even or vice versa, $a_i^{(j)}=\sum\limits_{k=j+1}^\ell d_{i+j-2k}^{(k)}$. The only condition we need to verify is that $a_i^{(j)}\leq a_{i-1}^{(j-1)}$. This can be shown by induction on $\ell-j$: if $j=\ell$, then $a_i^{(j)}=0 \leq d_{i-\ell-2}^{(\ell)} = a_{i-1}^{(j-1)}$. For arbitrary $j$, we assume inductively that $a_{i-1}^{(j+1)}\leq a_{i-2}^{(j)}$. We know that $d_{i-j-2}^{(j+1)} \leq d_{i-j-2}^{(j)}$, and adding these inequalities gives $a_i^{(j)}\leq a_{i-1}^{(j-1)}$.

Hence it suffices to count $(n+1)/2$-tuples of integer partitions
\begin{equation}
\delta_i: d_i^{(1)}\geq d_i^{(2)} \geq d_i^{(3)} \cdots
\end{equation}
for $i$ modulo $n+1$ odd, with $\sum\limits_{i, j} d_i^{(j)}=a(n+1)/2$, such that:
\begin{condition}
For fixed $j$, the $d_i^{(j)}$ are either all even or all odd.
\end{condition}
\begin{condition} \label{tele} 
$\sum\limits_j d_{i-2j}^{(j)}$ is the same for all $i$.
\end{condition}
We will use the notation $\lambda+\mu$ for adding two partitions entry-by-entry, $c \lambda$ for multiplying all entries by a constant, $\ell(\lambda)$ for the length of a partition, and $\lambda^*$ for the conjugate partition.

First I claim that there exists a unique strictly decreasing partition $\gamma$ such that for all $i$, there exists a partition $\tilde{\delta}_i$ with $\delta_i=2 \tilde{\delta}_i + \gamma^*$. We may take $\gamma$ to be the set $\lbrace j : d_i^{(j)} \text{ odd}\rbrace$, in decreasing order. If $\gamma_1$ and $\gamma_2$ have this same property, then $\gamma_1^*+\gamma_2^*$ has all even entries, and, since $\gamma_1$ and $\gamma_2$ are strictly decreasing, this implies that they are equal. 

Since the generating function of strictly decreasing partitions is the same as the generating function of odd partitions, $\prod\limits_{k=0}^{\infty}(1-x^{2k+1})^{-1}$, the first factor of $R^{\flat}_{\text{diag}}(x)$ will account for the choice of $\gamma$.

Now it suffices to count $(n+1)/2$-tuples of integer partitions $\tilde{\delta}_i: \tilde{d}_i^{(1)}\geq \tilde{d}_i^{(2)} \geq \tilde{d}_i^{(3)} \cdots$ satisfying condition (\ref{tele}). I claim that any partition $\tilde{\delta}_i$ can be written uniquely as $\sum\limits_{k=1}^{(n+1)/2} \tilde{\delta}_{i, k}^*$ where $\tilde{\delta}_{i, k}$ is a partition all of whose entries are congruent to $k$ modulo $(n+1)/2$. If $k' \equiv k \mod (n+1)/2$, then the multiplicity of the entry $k'$ in the partition $\tilde{\delta}_{i, k}$ will be $\tilde{d}_i^{(k')}-\tilde{d}_i^{(k'+1)}$. 

Under this decomposition, we consider the contribution of some fixed $\tilde{\delta}_{i_0, k}^*$ to each sum $\sum\limits_j d_{i-2j}^{(j)}$ in condition (\ref{tele}). $\tilde{\delta}_{i_0, k}^*$ is constructed to contribute the same amount to each sum, except for an additional $\ell(\tilde{\delta}_{i_0, k})$ if $i-i_0$ is between $2$ and $2k$ modulo $n+1$. Thus condition (\ref{tele}) means that 
\begin{equation}
\sum_{k=1}^{(n+1)/2} \sum_{\substack{i-2k \leq i_0 \leq i-2 \\ i_0 \text{ odd}}} \ell(\tilde{\delta}_{i_0, k})
\end{equation}
is the same for all $i$.

Now let us examine the remaining factors of $R^{\flat}$. For simplicity I will temporarily replace $x_i^2$ with $x_i$ and rewrite these remaining factors as
\begin{equation}\label{factors} 
\prod_{m=0}^{\infty} \prod_{k=1}^{(n+1)/2} \prod_{\substack{i_0 \mod n+1 \\ i_0 \text{ odd}}} (1-(x_1 x_3 \cdots x_n)^m(x_{i_0+2} x_{i_0+4}\cdots x_{i_0+2k}))^{-1}
\end{equation}
We think of the factor $\prod\limits_{m=0}^{\infty} (1-(x_1 x_3 \cdots x_n)^m(x_{i_0+2} x_{i_0+4}\cdots x_{i_0+2k}))^{-1}$ as generating the partitions $\tilde{\delta}_{i_0, k}$. If we expand this factor as a power series, the coefficient of a term 
\begin{equation} 
(x_{i_0+2} x_{i_0+4}\cdots x_{i_0+2k})^{\ell} (x_1 x_3 \cdots x_n)^w 
\end{equation} 
counts the number of partitions $\tilde{\delta}_{i_0, k}$ of $\ell k+w(n+1)/2$ with length $\ell$ and all entries congruent to $k$ modulo $(n+1)/2$. The exponent of the variable $x_i$ in this term is the same for all $i$, except for an additional $\ell(\tilde{\delta}_{i_0, k})$ if $i-i_0$ is between $2$ and $2k$ modulo $n+1$. Thus a diagonal term in the power series of (\ref{factors}) represents a set of $\tilde{\delta}_{i_0, k}$ satisfying the condition (\ref{tele}).

It follows that the generating function of $(n+1)/2$-tuples of partitions
\begin{equation}
(\delta_i)=(\gamma^*+2\tilde{\delta}_i)=(\gamma^*+\sum_{k=1}^{(n+1)/2} 2\tilde{\delta}_{i, k}^*)
\end{equation}
is precisely $R^{\flat}_{\text{diag}}(x)$.

\end{proof}

This completes the verification of equation (\ref{residue}).

\singlespace
\bibliography{bib}

\providecommand{\bysame}{\leavevmode\hbox to3em{\hrulefill}\thinspace}
\providecommand{\MR}{\relax\ifhmode\unskip\space\fi MR }
\providecommand{\MRhref}[2]{%
  \href{http://www.ams.org/mathscinet-getitem?mr=#1}{#2}
}
\providecommand{\href}[2]{#2}
\begin{thebibliography}{10}

\bibitem{B}
B.J. Birch, \emph{How the number of points of an elliptic curve over a fixed
  prime field varies}, J. London Math. Soc. \textbf{s1-43} (1968), no.~1,
  57--60.

\bibitem{Bo}
N.~Bourbaki, \emph{Lie groups and {L}ie algebras, chapters 4-6}, Elements of
  Mathematics, Springer-Verlag, Berlin, 2002.

\bibitem{BGKP}
A.~Braverman, H.~Garland, D.~Kazhdan, and M.~Patnaik, \emph{An affine
  {G}indikin-{K}arpelevic formula}, preprint.

\bibitem{BK}
A.~Braverman and D.~Kazhdan, \emph{The spherical {H}ecke algebra for affine
  {K}ac-{M}oody groups {I}}, Ann. of Math \textbf{(2) 174} (2011), no.~3,
  1603--1642.

\bibitem{BBF1}
B.~Brubaker, D.~Bump, and S.~Friedberg, \emph{Weyl group multiple {D}irichlet
  series, {E}isenstein series and crystal bases}, Ann. of Math. \textbf{(2)
  173} (2011), no.~2, 1081--1120.

\bibitem{BBF2}
\bysame, \emph{Weyl group multiple {D}irichlet series: Type {A} combinatorial
  theory}, Ann. of Math. Studies, vol. 175, Princeton University Press,
  Princeton, NJ, 2011.

\bibitem{BrFH}
B.~Brubaker, S.~Friedberg, and J.~Hoffstein, \emph{Cubic twists of {$GL(2)$}
  automorphic {L}-functions}, Invent. Math. \textbf{160} (2005), no.~1, 31--58.

\bibitem{BD}
A.~Bucur and A.~Diaconu, \emph{Moments of quadratic {D}irichlet {L}-functions
  over rational function fields}, Moscow Math. J. \textbf{10} (2010), no.~3,
  485--517.

\bibitem{BFH1}
D.~Bump, S.~Friedberg, and J.~Hoffstein, \emph{A nonvanishing theorem for
  derivatives of automorphic {L}-functions with applications to elliptic
  curves}, Bulletin of the Amer. Math. Soc. \textbf{21} (1989), no.~1, 89--95.

\bibitem{BFH3}
\bysame, \emph{On some applications of automorphic forms to number theory},
  Bulletin of the Amer. Math. Soc. \textbf{33} (1996), no.~2, 157--175.

\bibitem{BFH2}
\bysame, \emph{Sums of twisted {$GL(3)$} automorphic {L}-functions},
  Contributions to automorphic forms, geometry, and number theory, Johns
  Hopkins Univ. Press, Baltimore, MD, 2004, pp.~131--162.

\bibitem{CG1}
G.~Chinta and P.E. Gunnells, \emph{Weyl group multiple {D}irichlet series
  constructed from quadratic characters}, Invent. Math. \textbf{167} (2007),
  no.~2, 327--353.

\bibitem{CG2}
\bysame, \emph{Constructing {W}eyl group multiple {D}irichlet series}, J. Amer.
  Math. Soc. \textbf{23} (2010), no.~1, 189--215.

\bibitem{CFKRS}
J.B. Conrey, D.W. Farmer, J.P. Keating, M.~Rubinstein, and N.~Snaith,
  \emph{Integral moments of {L}-functions}, Proc. London Math. Soc. \textbf{91}
  (2005), 33--104.

\bibitem{D1}
P.~Deligne, \emph{La conjecture de {W}eil: {I}}, Publications Mathématiques de
  l'IH\'{E}S \textbf{43} (1974), 273--307.

\bibitem{D2}
\bysame, \emph{La conjecture de {W}eil: {II}}, Publications Mathématiques de
  l'IH\'{E}S \textbf{52} (1980), 137--252.

\bibitem{DGH}
A.~Diaconu, D.~Goldfeld, and J.~Hoffstein, \emph{Multiple {D}irichlet series
  and moments of zeta and {L}-functions}, Compositio Math. \textbf{139} (2003),
  no.~3, 297--360.

\bibitem{DP}
A.~Diaconu and V.~Pasol, \emph{Trace formulas, character sums, and multiple
  {D}irichlet series}, preprint.

\bibitem{FF1}
B.~Fisher and S.~Friedberg, \emph{Sums of twisted {$GL(2)$} {L}-functions over
  function fields}, Duke Math. J. \textbf{117} (2003), no.~3, 543--570.

\bibitem{FF2}
\bysame, \emph{Double {D}irichlet series over function fields}, Compositio
  Math. \textbf{140} (2004), no.~3, 613--630.

\bibitem{FZ}
S.~Friedberg and L.~Zhang, \emph{Eisenstein series on covers of odd orthogonal
  groups}, preprint.

\bibitem{G}
H.~Garland, \emph{Certain {E}isenstein series on loop groups: convergence and
  the constant term}, Algebraic Groups and Arithmetic, Tata Inst. Fund. Res.,
  Mumbai, 2004, pp.~275--319.

\bibitem{GMP}
H.~Garland, M.~Patnaik, and S.~Miller, \emph{Entirety of cuspidal {E}isenstein
  series on loop groups}, preprint.

\bibitem{GH}
D.~Goldfeld and J.~Hoffstein, \emph{Eisenstein series of 1/2 integral weight
  and the mean value of real {D}irichlet {L}-series}, Invent. Math. \textbf{80}
  (1985), no.~2, 185--208.

\bibitem{HR}
J.~Hoffstein and M.~Rosen, \emph{Average values of {L}-series in function
  fields}, J. Reine Agnew. Math. \textbf{426} (1992), 117--150.

\bibitem{H}
J.E. Humphreys, \emph{Reflection groups and {C}oxeter groups}, Cambridge
  Studies in Advanced Mathematics, vol.~29, Cambridge University Press,
  Cambridge, UK, 1990.

\bibitem{J}
M.~Jutila, \emph{On the mean values of $l(\frac{1}{2}, \chi)$ for real
  characters}, Analysis \textbf{1} (1981), no.~2, 149--161.

\bibitem{KS}
N.~Katz and P.~Sarnak, \emph{Random matrices, {F}robenius eigenvalues, and
  monodromy}, Amer. Math. Soc. Colloquium Publications, Providence, RI, 1999.

\bibitem{LZ}
K.H. Lee and Y.~Zhang, \emph{Weyl group multiple {D}irichlet series for
  symmetrizable {K}ac-{M}oody root systems}, preprint.

\bibitem{M}
I.G. Macdonald, \emph{Affine root systems and {D}edekind's $\eta$-function},
  Invent. Math. \textbf{15} (1972), 91--143.

\bibitem{R}
M.~Rosen, \emph{Number theory in function fields}, Graduate Texts in
  Mathematics, vol. 210, Springer-Verlag, New York, 2002.

\bibitem{S}
C.L. Siegel, \emph{The average measure of quadratic forms with given
  determinant and signature}, Gesammelte Abhandlungen II, Springer,
  Berlin-Heidelburg-New York, 1966, pp.~473--491.

\bibitem{So}
K.~Soundararajan, \emph{Nonvanishing of quadratic {D}irichlet {L}-functions at
  $s=\frac{1}{2}$}, Ann. of Math. \textbf{152} (2000), no.~2, 447--488.

\end{thebibliography}
\bibliographystyle{amsplain}
\addcontentsline{toc}{chapter}{Bibliography}

%\appendix
% \chapter*{Appendix 1: ...}
% \markboth{APPENDIX 1: ...}{}
% \label{chap:appendix}
% \input{appendix.tex}
% \addcontentsline{toc}{chapter}{Appendix I: ...}

\end{document}